\documentclass[journal,10pt,twocolumn,twoside]{IEEEtran}

\usepackage[utf8]{inputenc} 
\usepackage[T1]{fontenc}    
\usepackage{hyperref}       
\usepackage{url}            
\usepackage{booktabs}       
\usepackage{nicefrac}       
\usepackage{microtype}      
\usepackage{graphics}       
\usepackage{amssymb,amsmath,amscd,amsfonts,amsthm,bbm,mathrsfs}
\usepackage[noadjust]{cite}

\IEEEoverridecommandlockouts

\usepackage{algpseudocode}
\usepackage{stmaryrd,xcolor} 

\newtheorem{theorem}{Theorem}
\newtheorem{lemma}[theorem]{Lemma}
\newtheorem{prop}[theorem]{Proposition}

\newtheorem{assumption}{Assumption}

\newenvironment{example}[1][Example.]{\begin{trivlist}
\item[\hskip \labelsep {\bfseries #1}]}{\end{trivlist}}
\newenvironment{remark}[1][Remark.]{\begin{trivlist}
\item[\hskip \labelsep {\bfseries #1}]}{\end{trivlist}}

\makeatletter
\newcommand\footnoteref[1]{\protected@xdef\@thefnmark{\ref{#1}}\@footnotemark}
\makeatother

\newcommand{\bR}{\mathbb{R}}

\newcommand{\bN}{\mathbb{N}}
\newcommand{\bE}{\mathbb{E}}
\newcommand{\bP}{\mathbb{P}}
\newcommand{\cF}{\mathcal{F}}
\newcommand{\cG}{\mathcal{G}}
\newcommand{\cV}{\mathcal{V}}
\newcommand{\cE}{\mathcal{E}}
\newcommand{\cK}{\mathcal{K}}
\newcommand{\cL}{\mathcal{L}}

\newcommand{\cX}{\mathcal{X}}
\newcommand{\prox}{\mathop{\mathrm{prox}}\nolimits}
\newcommand{\ps}[1]{\langle #1 \rangle}
\newcommand{\1}{\mathbbm 1}

\newcommand{\mZ}{{\mathcal Z}}

\newcommand{\sI}{{\mathsf I}}
\newcommand{\sM}{{\mathsf M}}
\newcommand{\sx}{{\mathsf x}}
\newcommand{\sX}{{\mathsf X}}

\newcommand{\sT}{{\mathsf T}}
\newcommand{\su}{{\mathsf u}}
\newcommand{\sz}{{\mathsf z}}

\newcommand{\mcB}{{\mathscr B}}

\newcommand{\mcF}{{\mathscr F}}

\usepackage[textwidth=2cm, textsize=footnotesize]{todonotes}  
\newcommand{\as}[1]{{\color{black} #1}}

\title{Snake: a Stochastic Proximal Gradient Algorithm for 
Regularized Problems over Large Graphs} 

\author{Adil Salim, Pascal Bianchi and Walid Hachem 
\thanks{A.~Salim and P.~Bianchi are with the 
LTCI, Télécom ParisTech, Université Paris-Saclay, 75013, Paris, France 
(\texttt{adil.salim, pascal.bianchi@telecom-paristech.fr}).}
\thanks{W.~Hachem is with the CNRS / LIGM (UMR 8049), Universit\'e Paris-Est 
Marne-la-Vallée, France (\texttt{walid.hachem@u-pem.fr}).}
\thanks{This work was supported by the Agence Nationale pour la Recherche,
France, (ODISSEE project, ANR-13-ASTR-0030) and by the Labex Digiteo-DigiCosme
(OPALE project), Universit\'e Paris-Saclay.} 
\thanks{The authors are grateful to TeraLab DataScience for their material support.}
}

\date{\today} 

\begin{document}
\maketitle

\begin{abstract}
A regularized optimization problem over a large unstructured graph is studied,
where the regularization term is tied to the graph geometry.  Typical
regularization examples include the total variation and the Laplacian
regularizations over the graph. When applying the proximal gradient algorithm
to solve this problem, there exist quite affordable methods to implement the
proximity operator (backward step) in the special case where the graph is a
simple path without loops.  In this paper, an algorithm, referred to as
``Snake'', is proposed to solve such regularized problems over general graphs,
by taking benefit of these fast methods.  The algorithm consists in properly
selecting random simple paths in the graph and performing the proximal gradient
algorithm over these simple paths.  This algorithm is an instance of a new general
stochastic proximal gradient algorithm, whose convergence is proven.
Applications to trend filtering and graph inpainting are provided among others.
Numerical experiments are conducted over large graphs.  
\end{abstract}

\section{Introduction}
Many applications in the fields of machine learning
\cite{el2016asymptotic,zhu2003semi,hallac2015network},
signal and image restoration
\cite{chambolle2010introduction,hinterberger2003tube,harchaoui2012multiple}, or
trend filtering
\cite{tibshirani2014adaptive,wang2014trend,padilla2016dfs,hutter2016optimal,landrieu2016cut,tansey2015fast}
require the solution of the following optimization problem. On an undirected 
graph $G=(V,E)$ with no self loops, where $V=\{1,\ldots, N\}$ represents a set of $N$ nodes ($N\in {\mathbb N}^*$) and $E$ is the set of edges, find 
\begin{equation}
\min_{x\in \bR^V} F(x) +R(x,\phi),
\label{eq:prox}
\end{equation} 
where $F$ is a convex and differentiable function on $\bR^V$ representing a 
data fitting term, and the function $x\mapsto R(x,\phi)$ represents a 
regularization term of the form 
\[
R(x,\phi) = \sum_{\{i,j\}\in E } \phi_{\{i,j\}}(x(i),x(j)) \, ,
\]
where $\phi = (\phi_e)_{e\in E}$ is a family of convex and symmetric 
$\bR^2\to\bR$ functions. The regularization term $R(x,\phi)$ will be called a $\phi$-regularization in the sequel. These $\phi$-regularizations often promotes 
the sparsity or the smoothness of the solution. 
For instance, when $\phi_{e}(x,x') = w_{e}|x - x'|$ where 
$w = (w_{e})_{e\in E}$ is a vector of positive weights, the 
function $R(\cdot,\phi)$ coincides with the weighted Total Variation (TV) norm. 
This 
kind of regularization is often used in programming problems over a graph 
which are intended to recover a piecewise constant signal across adjacent nodes
\cite{padilla2016dfs,wang2014trend,hutter2016optimal,landrieu2016cut,
tansey2015fast,barberoTV14,ben2015robust,chen2014signal}.
Another example is the Laplacian regularization 
$\phi_{e}(x, x') = (x - x')^2,$ or its normalized version obtained by 
rescaling $x$ and $x'$ by the degrees of each node in $e$ respectively.
Laplacian regularization tends to 
smoothen the solution in accordance with the graph geometry
\cite{el2016asymptotic,zhu2003semi}.

The Forward-Backward (or proximal gradient) algorithm is one of the most
popular approaches towards solving Problem~\eqref{eq:prox}. This algorithm produces the sequence of iterates 
\begin{equation}
\label{fb-deter} 
x_{n+1} = \prox_{\gamma R(\cdot,\phi)}(x_n - \gamma \nabla F (x_n))\,,
\end{equation} 
where $\gamma > 0$ is a fixed step, and where 
\[ 
\prox_{g}(y) = \arg\min_{x} \left( g(x) + \frac 12\|x-y\|^2 \right) 
\]
is the well-known proximity operator applied to the proper, lower 
semicontinuous (lsc), and convex function $g$ (here $\|\cdot\|$ is the standard 
Euclidean norm). When $F$ satisfies a smoothness assumption, and when $\gamma$ 
is small enough, it is indeed well-known that the sequence $(x_n)$ converges to
a minimizer of~\eqref{eq:prox}, assuming this minimizer exists. 

Implementing the proximal gradient algorithm requires the computation of the
proximity operator applied to $R(\cdot,\phi)$ at each iteration.  When $N$ is
large, this computation is in general affordable only when the graph exhibits a
simple structure. For instance, when $R(\cdot,\phi)$ is the TV norm, the
so-called \emph{taut-string} algorithm is an efficient algorithm for computing
the proximity operator when the graph is one-dimensional
(1D)~\cite{condat2013direct} (see Figure~\ref{fig:graph-presentation}) or when it is a two-dimensional (2D) regular
grid~\cite{barberoTV14}.  Similar observations can be made for the Laplacian
regularization~\cite{graham1997spectral}, where, \emph{e.g.}, the discrete
cosine transform can be implemented. When the graph is large and unstructured,
these algorithms cannot be used, and the computation of the proximity operator
is more difficult (\cite{wang2014trend, spielman2010algorithms}). 

This problem is addressed in this paper. Towards obtaining a simple algorithm,
we first express the functions $F(\cdot)$ and $R(\cdot, \phi)$ as the
expectations of functions defined on a random walks in the graph, paving the
way for a \emph{randomized} version of the proximal gradient algorithm.
Stochastic online algorithms in the spirit of this algorithm are often
considered as simple and reliable procedures for solving high dimensional
machine learning problems, including in the situations where the randomness is
not inherent to these problems~\cite{bottou2010large, bottou2016optimization}.  One specificity of the algorithm
developed here lies in that it reconciles two requirements: on the one hand,
the random versions of $R(\cdot, \phi)$ should be defined on \emph{simple
paths}, \emph{i.e.}, on walks without loops (see Figure~\ref{fig:graph-presentation}), in a way to make benefit of the
power of the existing fast algorithms for computing the proximity operator.
Owing to the existence of a procedure for selecting these simple paths, we term
our algorithm as the ``Snake'' algorithm. On the other hand, the expectations
of the functions handled by the optimization algorithm coincide with $F(\cdot)$
and $R(\cdot,\phi)$ respectively (up to a multiplicative constant), in such a
way that the algorithm does not introduce any bias on the estimates. 

\begin{figure}[ht!]
\[
  \begin{array}{cc}
  \includegraphics[width=.5\linewidth]{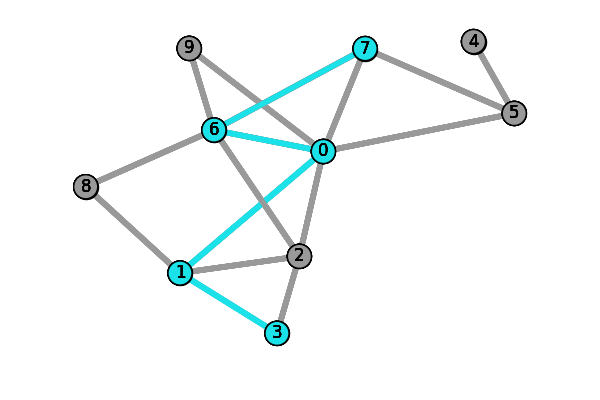} &
  \includegraphics[width=.5\linewidth]{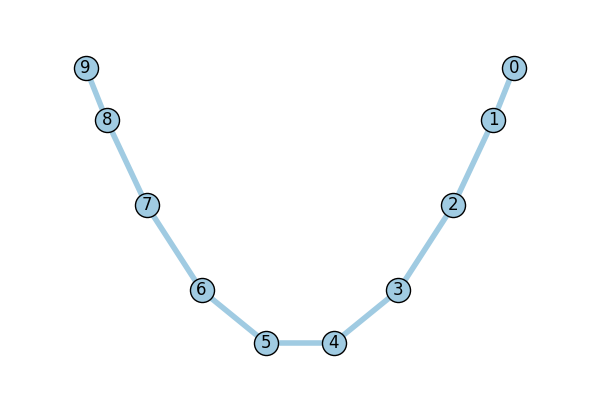}
  \end{array}
\]
\caption{Left: General graph on which is colored the simple path 3-1-0-6-7. Right: 1D-graph.}
\label{fig:graph-presentation}
\end{figure}

There often exists efficient methods to compute the proximity operator of $\phi$-regularization over 1D-graphs. The algorithm Snake randomly selects simple paths in a general graph in order to apply the latter 1D efficient methods over a general graph. 

Actually, the algorithm Snake will be an instance of a new general stochastic
approximation algorithm that we develop in this paper. In some aspects, this
last general stochastic approximation algorithm is itself a generalization of the random Forward-Backward algorithm studied in \cite{bia-hac-16}. 

Before presenting our approach, we provide an overview of the literature
dealing with our problem. First consider the case where
$R(\cdot,\phi)$ coincides with the TV norm. As said above, fast methods exist when
the graph has a simple structure. {We refer the reader to
\cite{barberoTV14} for an overview of iterative solvers of
Problem~\eqref{eq:prox} in these cases.} In \cite{johnson2013dynamic}, the author
introduces a dynamical programming method to compute the proximity operator on
a 1D-graph with a complexity of order $\mathcal O(N)$. 
Still in the 1D case, Condat~\cite{condat2013direct}
revisited recently an algorithm that is due to Mammen and Van De Geer
\cite{mammen1997locally} referred to as the taut-string
algorithm. The complexity of this algorithm is $\mathcal O(N^2)$ in the 
worst-case scenario, and $\mathcal O(N)$ in the most realistic cases. The taut-string algorithm is linked to a total
variation regularized problem in
\cite{davies2001local}. 
This algorithm is generalized to 2D-grids, weighted TV norms and
$\ell^{p}$ TV norms by Barbero and Sra in \cite{barberoTV14}. To
generalize to 2D-grids, the TV regularization can be
written as a sum of two terms on which one can apply 1D methods,
according to \cite{combettes2009iterative} and
\cite{jegelka2013reflection}. Over general graphs, there is no
immediate way to generalize the taut string method. 
The problem of
computing the TV-proximity operator over a general graph is addressed in
\cite{wang2014trend}. 

{The authors of~\cite{wang2014trend} suggest to solve the problem
  using a projected Newton algorithm applied to the dual problem. It
  is observed that, empirically, this methods performs better than
  other concurrent approaches.  As a matter of fact, this statement
  holds when the graph has a moderate size.  As far as large graphs
  are concerned, the iteration complexity of the projected Newton
  method can be a bottleneck. To address this problem, the authors of
\cite{ben2015robust} and~\cite{hallac2015network} propose
  to solve the problem distributively over the nodes using the
  Alternating Direction Method of Multipliers (ADMM).

In~\cite{tansey2015fast} the authors propose to compute a decomposition of the graph in 1D-graphs and then solve Problem~\eqref{eq:prox} by means of the TV-proximity operators over these 1D-graphs. Although the decomposition of the graph is fast in many applications, the algorithm~\cite{tansey2015fast} relies on an offline decomposition of the whole graph that needs a global knowledge of the graph topology. The Snake algorithm obtains this decomposition online. 
In \cite{landrieu2016cut}, the authors propose a working set strategy to compute the TV-proximity operator. At each iteration, the graph is cut in two well-chosen subgraphs and a reduced problem of~\eqref{eq:prox} is deduced from this cut. The reduced problem is then solved efficiently. This method has shown speed-ups when $G$ is an image (\textit{i.e} a two dimensional grid). Although the decomposition of the graph is not done during the preprocessing time, the algorithm~\cite{landrieu2016cut} still needs a global knowledge of the graph topology during the iterations. On the contrary, the Snake algorithm only needs a local knowledge. 
Finally, in~\cite{padilla2016dfs}, the authors propose to replace the computation of the TV-proximity operator over the graph $G$ by the computation of the TV-proximity operator over an 1D-subgraph of $G$ well chosen. This produces an approximation of the solution whereas the Snake algorithm is proven to converge to the exact solution.
}

{
In the case where $R(\cdot,\phi)$ is the Laplacian regularization, 
the computation of the proximity operator of $R$ reduces to the
resolution of a linear system $(\mathcal{L}+\alpha I)x = b$ where
$\mathcal L$ is the Laplacian matrix of the graph $G$ and $I$ the
identity matrix. On an 1D-graph, the latter resolution can be done
efficiently and relies on an explicit diagonalization of $\mathcal{L}$
(\cite{graham1997spectral}) by means of the discrete cosine transform,
which take $\mathcal O(N \log (N))$ operations. Over general graphs,
the problem of computing the proximity operator of the Laplacian
regularization is introduced in~\cite{zhu2003semi}. There exist fast
algorithms to solve it due to Spielman and
Teng~\cite{spielman2014nearly}. They are based on recursively
preconditioning the conjugate gradient method using graph theoretical
results~\cite{spielman2010algorithms}. Nevertheless, the preconditionning phase
which may be demanding over very large graphs. Compared
to~\cite{spielman2010algorithms}, our online method Snake requires no
preprocessing step.  }
\section{Outline of the approach and paper organization}
\label{sec:outline} 

The starting point of our approach is a new stochastic optimization 
algorithm that has its own interest. This algorithm will be presented 
succinctly here, and more rigorously in Sec.~\ref{sec:as} below. 
Given an integer $L > 0$, let $\xi = (\xi^1,\ldots, \xi^L)$ be a random vector
where the $\xi^i$ are valued in some measurable space. Consider the problem 
\begin{equation}
\label{eq:probrandom} 
  \min_{x}\  \sum_{i=1}^L\bE_\xi \left[ f_i(x,\xi^i)+g_i(x,\xi^i)\right]
\end{equation} 
where the $f_i(\cdot, \xi^i)$ are convex and differentiable, and the 
$g_i(\cdot, \xi^i)$ are convex. Given $\gamma > 0$, define the operator 
$\sT_{\gamma,i}(x,s) = \prox_{\gamma g_i(\cdot,s)}(x-\gamma \nabla f_i(x,s))$.
Given a sequence $(\xi_n)$ of independent copies of $\xi$, and a sequence 
of positive steps $(\gamma_n) \in \ell^2\setminus\ell^1$, we consider the  
algorithm 
\begin{equation}
\label{eq:conc}
x_{n+1} = \sT_{\gamma_{n+1}}(x_n, \xi_{n+1})\, ,
\end{equation}
where 
\[
\sT_{\gamma}(\cdot, (s^1,\ldots, s^L)) = 
\sT_{\gamma,L}(\cdot,s^L) \circ\dots\circ \sT_{\gamma,1}(\cdot,s^1)
\]
and where $\circ$ stands for the composition of functions: $f \circ g (\cdot) = f(g(\cdot))$. In other words, an iteration of this algorithm consists in the composition of $L$ random 
proximal gradient iterations. The case where $L=1$ was treated in 
\cite{bia-hac-16}. 

Assuming that the set of minimizers of the problem is non empty,
Th.~\ref{th:main} below states that the sequence $(x_n)$ converges almost
surely to a (possibly random) point of this set. The proof of this theorem
which is rather technical is deferred to Sec.~\ref{sec:prf-main}. It follows
the same canvas as the approach of \cite{bia-hac-16}, with the difference that
we are now dealing with possibly different functions $(f_i, g_i)$ and
non-independent noises $\xi^i$ for $i\in\{1,\ldots, L\}$. 

We now want to exploit this stochastic algorithm to develop a simple procedure 
leading to a solution of Problem~\eqref{eq:prox}. This will be done in
Sec.~\ref{sec:algo} and will lead to the Snake algorithm.  The first step is to
express the function $R(\cdot,\phi)$ as the expectation of a function with
respect to a finite random walk.  Given an integer $M > 0$ and a finite walk 
$s = (v_0, v_1, \ldots, v_M)$ of length $M$ on the graph $G$, where 
$v_i \in V$ and $\{ v_i, v_{i+1} \} \in E$, write 
\[ 
R(x, \phi_s) = \sum_{i=1}^M \phi_{\{v_{i-1},
v_i\}}(x(v_{i-1}), x({v_i})) \, .  
\]
Now, pick a node at random with a probability proportional to the degree 
(\emph{i.e.}, the number of neighbors) of this node. Once this node has been 
chosen, pick another one at random uniformly among the neighbors of the first 
node. Repeat the process of choosing neighbors $M$ times, and denote as 
$\xi \in V^{M+1}$ the random walk thus obtained. 
With this construction, we get that
$\frac{1}{|E|}R(x,\phi) = \frac{1}{M} \bE_\xi [ R(x, \phi_{\xi}) ]$
using some elementary Markov chain formalism (see Prop.~\ref{prop:tvE} below).

In these conditions, a first attempt of the use of Algorithm~\eqref{eq:conc} 
is to consider Problem~\eqref{eq:prox} as an instance of
Problem~\eqref{eq:probrandom} with $L = 1$, $f_1(x,\xi) = \frac{1}{|E|} F(x)$, and 
$g_1(x,\xi) = \frac{1}{M} R(x,\phi_\xi)$. Given an independent sequence $(\xi_n)$ of 
walks having the same law as $\xi$ and a sequence $(\gamma_n)$ of steps in 
$\ell^2\setminus\ell^1$, Algorithm~\ref{eq:conc} boils down to the stochastic 
version of the proximal gradient algorithm
\begin{equation}
\label{eq:proxL=1} 
x_{n+1} = \prox_{\gamma_{n+1} \frac{1}{M} R(\cdot,\phi_{\xi_{n+1}})}
    (x_n-\gamma_{n+1} \frac{1}{|E|} \nabla F(x_n)) \, .  
\end{equation} 
By Th.~\ref{th:main} (or by \cite{bia-hac-16}), the iterates $x_n$ converge 
almost surely to a solution of Problem~\eqref{eq:prox}.

However, although simpler than the deterministic algorithm~\eqref{fb-deter},
this algorithm is still difficult to implement for many regularization
functions. As said in the introduction, the walk $\xi$ is often required to be a
simple path. Obviously, the walk generation mechanism described above does not
prevent $\xi$ from having repeated nodes.  A first way to circumvent this
problem would be to generate $\xi$ as a loop-erased walk on the graph.
Unfortunately, the evaluation of the corresponding distribution is notoriously
difficult. The generalization of Prop.~\ref{prop:tvE} to loop-erased walks is
far from being immediate.

As an alternative, we identify the walk $\xi$ with the concatenation of at most
$M$ simple paths of maximal length that we denote as $\xi^1,\ldots, \xi^M$,
these random variables being valued in the space of all walks in $G$ of length
at most $M$:
$$
\xi = (\xi^1,\xi^2,\ldots, \xi^M)\,.
$$
Here, in the most frequent case where the number of simple paths
is strictly less than $M$, the last $\xi^i$'s are conventionally set to a trivial walk, \textit{i.e} a walk with one node and no edge.  
We also
denote as $\ell(\xi^i)$ the length of the simple path $\xi^i$, \emph{i.e.}, the number
of edges in $\xi^i$. We now choose $L=M$, and for $i=1,\ldots, L$, we set
$f_i(x,\xi^i) = \frac{\ell(\xi^i)}{L |E|} F(x)$ and $g_i(x,\xi^i) = \frac{1}{L} R(x, \phi_{\xi^i})$
if $\ell(\xi^i) > 0$, and $f_i(x,\xi^i) = g_i(x,\xi^i) = 0$ otherwise. With
this construction, we show in Sec.~\ref{sec:algo} that 
$\frac{1}{|E|} (F(x) + R(x,\phi)) = \sum_{i=1}^L \bE_\xi [ f_i(x,\xi^i) + g_i(x, \xi^i) ]$
and that the functions $f_i$ and $g_i$ fulfill the general assumptions required
for the Algorithm~\eqref{eq:conc} to converge to a solution of
Problem~\eqref{eq:prox}.  In summary, at each iteration, we pick up a random
walk of length $L$ according to the procedure described above, we split it into
simple paths of maximal length, then we successively apply the proximal
gradient algorithm 
to these simple paths. 

After recalling the contexts of the taut-string and the Laplacian
regularization algorithms (Sec.~\ref{sec:1D}), we simulate
Algorithm~\eqref{eq:conc} in several application contexts. First, we study the so called graph trend filtering \cite{wang2014trend}, with the parameter $k$ defined in~\cite{wang2014trend} set to one. Then, we consider the graph inpainting problem~\cite{chen2014signal,el2016asymptotic,zhu2003semi}. These contexts are the
purpose of Sec.~\ref{sec:examples}. Finally, a conclusion and some future
research directions are provided in Sec.~\ref{sec-conc}. 

\section{A General Stochastic Proximal Gradient Algorithm}
\label{sec:as}

\emph{Notations.}  We denote by $(\Omega,\cF,\bP)$ a probability space
and by $\bE [\cdot]$ the corresponding expectation.  We let $(\Xi,\cX)$ be an
arbitrary measurable space.  We denote $\sX$ some Euclidean space and
by $\mcB(\sX)$ its Borel $\sigma$-field.  A mapping $f:\sX\times
\Xi\to \bR$ is called a normal convex integrand if $f$ is
$\mcB(\sX)\otimes \cX$-measurable and if $f(\,.\,,s)$ is convex for
all $s\in \Xi$ \cite{roc-69(mes)}.

\subsection{Problem and General Algorithm}

In this section, we consider the general problem 
\begin{equation}
  \label{eq:general}
  \min_{x\in \sX}\  \sum_{i=1}^L\bE\left[ f_i(x,\xi^i)+g_i(x,\xi^i)\right]
\end{equation}
where $L$ is a positive integer, the $\xi^i:\Omega\to \Xi$ are 
random variables (r.v.), and the functions $f_i: {\sX}\times \Xi\to \bR$ and 
$g_i: {\sX}\times \Xi\to \bR$ satisfy the following assumption:
\begin{assumption} 
\label{fg}
The following holds for all $i\in \{1,\ldots, L\}$:
  \begin{enumerate}
  \item  The $f_i$ and $g_i$ are normal convex integrands.
  \item\label{f,g:integ}
   For every $x\in \sX$, $\bE[|f_i(x,\xi^i)|]<\infty$ and 
   $\bE[|g_i(x,\xi^i)|]<\infty$.
\item For every $s\in \Xi$, $f_i(\cdot,s)$ is differentiable. We denote as
   $\nabla f_i(\cdot,s)$ its gradient w.r.t.~the first variable. 
  \end{enumerate}
\end{assumption}

\begin{remark}
In this paper, we assume that the functions $g_i(\cdot, \xi)$ have a full domain
for almost all $\xi$. This assumption can be relaxed with some effort, along
the ideas developed in \cite{bia-hac-16}. 
\end{remark}

For every $i=1,\ldots,L$ and every $\gamma>0$, we introduce the mapping
$\sT_{\gamma,i}: {\sX}\times \Xi\to \sX$ defined by
\[
\sT_{\gamma,i}(x,s) = \prox_{\gamma g_i(\cdot,s)}(x-\gamma \nabla f_i(x,s))\,.
\] 
We define $\sT_{\gamma}: {\sX} \times \Xi^L \to \sX$ by
$$
\sT_{\gamma}(\cdot,(s^1,\ldots,s^L)) = 
 \sT_{\gamma,L}(\cdot,s^L)\circ\dots\circ \sT_{\gamma,1}(\cdot,s^1)\,.
$$
Let $\xi$ be the random vector $\xi = (\xi^1,\ldots,\xi^L)$ with values
in $\Xi^L$ and let $(\xi_n:n\in\bN^*)$ be a sequence of i.i.d. copies
of $\xi$, defined on the same probability space
$(\Omega,\cF,\bP)$. For all $n \in\bN^*,$ $\xi_n = (\xi_n^1,\ldots,\xi_n^L)$.  
Finally, let $(\gamma_n)$ be a positive sequence. Our aim is to analyze the 
convergence of the iterates $(x_n)$ recursively defined by:
\begin{equation}
x_{n+1} = \sT_{\gamma_{n+1}}(x_n,\xi_{n+1})\,,
\label{eq:spa}
\end{equation}
as well as the intermediate variables $\bar x_{n+1}^i$ ($i=0,\ldots,L$) defined 
by $\bar x_{n+1}^0 = x_n$, and  
\begin{equation}
\bar x_{n+1}^i = \sT_{\gamma_{n+1},i}(\bar x_{n+1}^{i-1},\xi_{n+1}^i)\, , 
\quad  i=1,\ldots,L\, . 
\label{eq:xintermediate}
\end{equation}
In particular, $x_{n+1} = x_{n+1}^{L} = T_{\gamma_{n+1}, L}(\bar x_{n+1}^{L-1},\xi_{n+1}^L)$.

In the special case where the functions $g_i$, $f_i$ are all constant with
respect to $s$ (the algorithm is deterministic), the above iterations were 
studied by Passty in \cite{pas-79}. In the special case where $L=1$,
the algorithm boils down to the stochastic Forward-Backward algorithm, whose
detailed convergence analysis can be found in \cite{bia-hac-16} (see also
\cite{bia-16}, and \cite{wan-ber-15} as an earlier work). In this case, the
iterates take the simpler form 
\begin{equation}
x_{n+1} = \prox_{\gamma_{n+1} g_1(\cdot,\xi_{n+1})}
       (x_n-\gamma_{n+1} \nabla f_1(x_n,\xi_{n+1}))\,,
\label{eq:sfb}
\end{equation}
and converge a.s.~to a minimizer of $\bE[f_1(x,\xi)+g_1(x_,\xi)]$ under the
convenient hypotheses. 

It is worth noting that the present algorithm~(\ref{eq:spa}) cannot be written
as an instance of (\ref{eq:sfb}). Indeed, the operator $\sT_\gamma$ is a composition
of $L$ (random) operators, whereas the stochastic forward backward algorithm (\ref{eq:sfb})
has a simpler structure. This composition raises technical difficulties that need to be specifically addressed.

\subsection{Almost sure convergence}

We make the following assumptions.
\begin{assumption}
  \label{step}
The positive sequence $(\gamma_n)$ satisfies the conditions
$$
\sum \gamma_n =+\infty\ \text{ and }\ \sum\gamma_n^2<\infty\, , 
$$
(\emph{i.e.}, $(\gamma_n) \in\ell^2\setminus\ell^1$).  
Moreover, $\frac{\gamma_{n+1}}{\gamma_n} \to 1$
\end{assumption}

\begin{assumption}
\label{f-Lip} The following holds for all $i\in \{1,\ldots, L\}$:
\begin{enumerate}
\item There exists a measurable map $K_i : \Xi \to \bR_+$ s.t. the following holds $\bP$-a.e.: for all $x,y$ in $\sX$,
$$
\| \nabla f_i(x,\xi_i) - \nabla f_i(y,\xi_i) \| \leq K_i(\xi_i) \| x - y \|\,.
$$
\item For all $\alpha > 0$, $\bE[K_i(\xi_i)^\alpha]<\infty$.
\end{enumerate}
\end{assumption} 

We denote by $\mZ$ the set of minimizers of Problem \eqref{eq:general}. Thanks
to Ass.~\ref{fg}, the qualification conditions hold, 
ensuring that a point $x_\star$ belongs to $\mZ$ \emph{iff} 
$$
0\in  \sum_{i=1}^L \nabla \bE[f_i(x_\star,\xi^i)] + \partial \bE[g_i(x_\star,\xi^i)]\,.
$$
The (sub)differential and the expectation operators can be interchanged 
\cite{roc-wet-82}, and the above optimality condition also reads
\begin{equation}
0\in  \sum_{i=1}^L \bE[\nabla f_i(x_\star,\xi^i)] + \bE[\partial g_i(x_\star,\xi^i)]\,,
\label{eq:optimality}
\end{equation}
where $\bE[\partial g_i(x_\star,\xi^i)]$ is the Aumann expectation of the
random set $\partial g_i(x_\star, \xi^i)$, defined as the set of expectations 
of the form $\bE[\varphi_i(\xi^i)]$, where $\varphi_i : \Xi \to \sX$ is a 
measurable map s.t. $\varphi_i(\xi^i)$ is integrable and
\begin{equation}
\varphi_i(\xi^i)\in \partial g_i(x_\star,\xi^i)\ \bP\text{-a.e.},\ \forall i .
\label{eq:representation}
\end{equation}
Therefore, the optimality condition~\eqref{eq:optimality}
means that there exist $L$ integrable mappings $\varphi_1,\ldots,\varphi_L$ satisfying \eqref{eq:representation} and s.t.
\begin{equation}
  \label{eq:optimality-2}
  0 =  \sum_{i=1}^L \bE[\nabla f_i(x_\star,\xi^i)] + \bE[\varphi_i(\xi^i)]\,.
\end{equation}
When (\ref{eq:representation})-(\ref{eq:optimality-2}) hold, we say that the family 
$(\nabla f_i(x_\star,\xi^i), \varphi_i(\xi^i))_{i=1,\ldots,L}$ is a \emph{representation} of the minimizer $x_\star$.
In addition, if for some $\alpha\geq 1$ and every $i=1,\ldots,L$,  $\bE[\|\nabla f_i(x_\star,\xi^i)\|^\alpha]<\infty$ and 
$\bE[\|\varphi(\xi^i)\|^\alpha]<\infty$, we say that the minimizer $x_\star$ admits a $\alpha$-integrable representation.
\begin{assumption}
\label{selec}
\begin{enumerate}
\item The set $\mZ$ is not empty. 
\item For every $x_\star \in \mZ$, there exists $\varepsilon>0$ s.t. $x_\star$ 
admits a $(2+\varepsilon)$-integrable representation 
$(\nabla f_i(x_\star,\xi^i), \varphi_i(\xi^i))_{i=1,\ldots,L}$. 
\end{enumerate}
\end{assumption} 
We denote by $\partial g_{i}^0(x, \xi^i)$ the least norm element in $\partial g_i(x,\xi^i)$. 
\begin{assumption}
\label{bnd-mom} 
For every compact set $\cK \subset \sX$, there exists $\eta > 0$ such 
that for all $i=1,\ldots,L$,
\[ 
\sup_{x\in \cK}\bE [\| \partial g_{i}^0(x, \xi^i) \|^{1+\eta}] < \infty\,.
\]
\end{assumption} 

We can now state the main result of this section, which will be proven in 
Sec.~\ref{sec:prf-main}. 
\begin{theorem}
\label{th:main} 
Let Ass.~\ref{fg}--\ref{bnd-mom} hold true. 
There exists a r.v. $X_\star$ s.t. $\bP(X_\star\in \mZ)=1$ and s.t.
$(x_n)$ converges a.s. to $X_\star$ as $n\to \infty$.
Moreover, for every $i=0,\ldots,L-1$, $\bar x_n^i$ converges a.s. to $X_\star$.
\end{theorem}

\section{The Snake Algorithm}
\label{sec:algo}

\subsection{Notations}
\label{sec:notations}

Let $\ell \geq 1$ be an integer. We refer to a walk of length $\ell$ over the
graph $G$ as a sequence $s=(v_0,v_1,\ldots,v_{\ell})$ in $V^{\ell+1}$ such
that for every $i=1,\ldots,\ell$, the pair $\{v_{i-1},v_{i}\}$ is an edge of
the graph. A walk of length zero is a single vertex. 

We shall often identify $s$ with the graph $\cG(s)$
whose vertices and edges are respectively given by the sets $\cV(s)=\{v_0,\ldots,v_{\ell}\}$ and 
$\cE(s) = \{\{v_0,v_1\},\ldots,\{v_{\ell-1},v_{\ell}\}\}$.

Let $L \geq 1.$ We denote by $\Xi$ the set of all walks over $G$ with length $\leq L.$ This is a finite set. Let $\cX$ be the set of all subsets of $\Xi.$ We consider the measurable space $(\Xi,\cX).$ 

Let $s=(v_0,v_1,\ldots,v_{\ell}) \in \Xi$ with $0 < \ell \leq L.$ We
abusively denote by $\phi_s$ the family of functions
$(\phi_{\{v_{i-1},v_{i}\}})_{i = 1,\ldots,\ell}.$ We refer to the
$\phi_{s}-$regularization of $x$ as the $\phi_s-$regularization on the
graph $s$ of the restriction of $x$ to $s$ that is
$$
R(x,\phi_s) = \sum_{i=1}^\ell \phi_{\{v_{i-1},v_{i}\}}(x(v_{i-1}), x(v_{i}))\,.
$$
Besides, $R(x,\phi_s)$ is defined to be $0$ if $s$ is a single vertex (that is $\ell = 0$).

We say that a walk is a \emph{simple path} if there is no repeated
node \textit{i.e}, all elements in $s$ are different or if $s$ is a single vertex.
Throughout the paper, we assume that when $s$ is a simple path, the
computation of $\prox_{R(.,\phi_s)}$ can be done easily.

\subsection{Writing the Regularization Function as an Expectation}

One key idea of this paper is to write the function $R(x,\phi)$ as an
expectation in order to use a stochastic approximation algorithm, as described
in Sec.~\ref{sec:as}.

Denote by $\deg(v)$ the degree of the node $v \in V$, \emph{i.e.}, the number
of neighbors of $v$ in $G$. Let $\pi$ be the probability measure on $V$ 
defined as 
\[
\pi(v) = \frac{\deg(v)}{2 |E|} , \quad v \in V \, . 
\] 
Define the probability transition kernel $P$ on $V^2$ as 
$P(v,w) = \1_{\{v,w\} \in E} / \deg(v)$ if $\deg(v) > 0$, and 
$P(v,w) = \1_{v=w}$ otherwise, where $\1$ is the indicator function.

We refer to a Markov chain (indexed by $\bN$) over $V$ with initial distribution $\pi$ and transition kernel $P$ as an infinite random walk over $G.$ 
Let $(v_{k})_{k \in \bN}$ be a infinite random walk over $G$ defined on the canonical probability space $(\Omega,\cF,\bP),$ with $\Omega = V^{\bN}.$ The first node $v_0$ of this walk is randomly chosen in $V$ according to the distribution $\pi.$ The other 
nodes are drawn recursively according to the conditional probability 
$\bP(v_{k+1} = w \, | \, v_k ) = P(v_k,w)$. In other words,
conditionally to $v_k$, the node $v_{k+1}$ is drawn uniformly from the 
neighborhood of $v_k$.  
Setting an integer $L \geq 1$, we define the random variable $\xi$ from $(v_{k})_{k \in \bN}$ as $\xi = (v_0,v_1,\ldots,v_{L}).$ 
\begin{prop}
\label{prop:tvE}
For every $x\in \bR^V$,
\begin{equation}
\frac{1}{|E|} R(x,\phi) = \frac{1}{L}\,\bE[R(x, \phi_\xi)]\,.\label{eq:tvE}
\end{equation}
\end{prop}

\begin{proof}
It is straightforward to show that $\pi$ is an 
invariant measure of the Markov chain $(v_k)_{k \in \bN}$.
Moreover, $\bP(v_{k} = w, v_{k-1} = v) = \pi(v) P(v,w) 
= \1_{\{v,w\}\in E} / (2|E|)$, leading to the identity
\[
\bE \left[\phi_{\{v_{k-1},v_{k}\}}(x(v_{k-1}),x(v_{k})) \right] = \frac{1}{|E|} R(x,\phi)\,,
\]
which completes the proof by symmetry of $\phi_e, \forall e \in E$. 
\end{proof}

This proposition shows that Problem~\eqref{eq:prox} is written equivalently
\begin{equation}
\label{eq:pbequiv}
\min_{x\in \bR^V} \frac{1}{|E|}F(x) + \,\bE[\frac{1}{L}R(x, \phi_\xi)].
\end{equation}
Hence, applying the stochastic proximal gradient algorithm to solve~\eqref{eq:pbequiv} leads to a new algorithm to solve~\eqref{eq:prox}, which was mentioned in Sec.~\ref{sec:outline}, Eq.~\eqref{eq:proxL=1}: 
\begin{equation}
\label{eq:fbsto}
x_{n+1} = \prox_{\gamma_{n+1} \frac{1}{L} R(\cdot,\phi_{\xi_{n+1}})}
    (x_n-\gamma_{n+1} \frac{1}{|E|} \nabla F(x_n)) \,.  
\end{equation}


Although the iteration complexity is reduced in~\eqref{eq:fbsto} compared to~\eqref{fb-deter}, the computation of the proximity operator of the $\phi$-regularization over the random subgraph $\xi_{n+1}$ in the algorithm~\eqref{eq:fbsto} can be difficult to implement. This is due to the
possible presence of loops in the random walk $\xi$. As an alternative, we split $\xi$ into several simple paths. We will then replace the proximity operator over $\xi$ by the series of the proximity operators over the simple paths induced by $\xi$, which are efficiently computable.

\subsection{Splitting $\xi$ into Simple Paths} 

Let $(v_{k})_{k \in \bN}$ be an infinite random walk on $(\Omega,\cF,\bP)$. 
We recursively define a sequence of stopping time $(\tau_i)_{i\in \bN}$ as $\tau_0=1$ and
for all $i\geq 0$,
$$
\tau_{i+1} = \min\{k \geq \tau_i:\,v_{k}\in \{v_{\tau_i - 1},\ldots,v_{k-1}\}\}
$$
if the above set is nonempty, and $\tau_{i+1}= +\infty$ otherwise. We now define the stopping times $t_i$ for all $i \in \bN$ as $t_i = \min(\tau_i,L+1).$ Finally, for all $i \in \bN^{*}$ we can consider the random variable $\xi^i$ on $(\Omega,\cF,\bP)$ with values in $(\Xi, \cX)$ defined by $$\xi^i = (v_{t_{i-1} - 1}, v_{t_{i-1}},\ldots, v_{t_{i} - 1}).$$
We denote by $N$ the smallest integer $n$ such that  $t_{n}=L+1$. We denote by $\ell(\xi^i)$ the length of the simple path $\xi^i.$

\begin{example}
  Given a graph with vertices $V=\{a,b,c,\ldots,z\}$ and a given edge set that is not useful to describe here, consider $\omega \in \Omega$ and the walk
  $ \xi(\omega) = (c,a,e,g,a,f,a,b,h)$ with length $L=8$. 
  Then, $t_0(\omega) = 1$, $t_1(\omega) =4$, $t_2(\omega)=6$, $t_3(\omega)=t_4(\omega)=\ldots=9,$
  and $\xi(\omega)$ can be decomposed into $N(\omega) = 3$ simple paths and we have
  $\xi^1(\omega) =  (c,a,e,g)$, $\xi^2(\omega) =  (g,a,f)$,  $\xi^3(\omega) =  (f,a,b,h)$ and $\xi^4(\omega) = \ldots = \xi^8(\omega) = (h).$
Their respective lengths are $\ell(\xi^1(\omega))=3$, $\ell(\xi^2(\omega))=2$, $\ell(\xi^3(\omega))=3$ and $\ell(\xi^i(\omega))=0$ for all $i = 4,\ldots,8$.
We identify $\xi(\omega)$ with $(\xi^1(\omega),\ldots,\xi^8(\omega)).$
\end{example}

It is worth noting that, by construction, $\xi^i$ is a simple path. Moreover, the following statements hold:
\begin{itemize}
\item We have $1 \leq N \leq L$ a.s.
\item These three events are equivalent for all $i$: \{$\xi^i$ is a single vertex\}, \{$\ell(\xi^i) = 0$\} and \{$i \geq N+1$\}
\item The last element of $\xi^N$ is a.s. $v_{L}$
\item $\sum_{i=1}^{L} \ell(\xi^i) = L$ a.s.
\end{itemize}

In the sequel, we identify the random vector $(\xi^1,\ldots,\xi^L)$ 
with the random variable $\xi = (v_0,\ldots,v_L).$ As a result, $\xi$ is seen as a r.v with values in $\Xi^L.$

Our notations are summarized in Table~\ref{tab:notations}.
\begin{table}[t]
  \centering
  \caption{Useful Notations} \label{tab:notations}
  \begin{tabular}[h]{|c|c|}
    \hline
    $G = (V,E)$ & Graph with no self-loop\\
    $s$ & walk on $G$ \\
    $(v_i)$ & infinite random walk \\
    $\xi = (\xi^1,\ldots,\xi^L)$ & random walk of length $L$ \\
    $\xi^i$ & random simple path \\
    $\ell(\xi^i)$ & length of $\xi^i$\\
    $R(x,\phi)$ & $\phi-$regularization of $x$ on  $G$ \\
    $R(x,\phi_s)$ & $\phi-$regularization of $x$ along the walk $s$ \\
    \hline
  \end{tabular}
\end{table}
  For every $i = 1,\ldots,L$, define the functions $f_i, g_i$ on
  $\bR^V\times \Xi$ in such a way that
  \begin{align}
    f_i(x,\xi^i)&= \frac{\ell(\xi^i)}{L |E|} F(x) \label{eq:fi} \\
    g_i(x,\xi^i)&= \frac{1}{L}\, R(x ,\phi_{\xi^i}) \,. \label{eq:gi} 
  \end{align}
  Note that when $i > N(\omega)$ then  $f_i(x,\xi^i(\omega)) = g_i(x,\xi^i(\omega)) = 0$.

\begin{prop}
\label{prop:pb3}
 For every $x\in \bR^V$, we have
\begin{equation}
\frac{1}{|E|}(F(x) + R(x,\phi)) = \sum_{i=1}^L\bE\left[ f_i(x,\xi^i)+g_i(x,\xi^i)\right]\,.
\label{eq:pb3}
\end{equation}
\end{prop}
\begin{proof}
For every $\omega \in \Omega$ and every $x\in \bR^V$,
\begin{equation*}
\frac{1}{L} R(x,\phi_{\xi(\omega)}) = \frac{1}{L} \sum_{i=1}^{N(\omega)}R (x,\phi_{\xi^i(\omega)}) = \sum_{i=1}^L g_i(x,\xi^i(\omega))\,.
\end{equation*}
Integrating, and using
Prop.~\ref{prop:tvE}, it follows that
$\sum_{i=1}^L \bE[g_i(x,\xi^i)] = \frac{1}{|E|} R(x,\phi)$.
Moreover, $\sum_{i=1}^L f_i(x,\xi^i(\omega)) = \frac{1}{|E|} F(x)$.
This completes the proof.
\end{proof}

\subsection{Main Algorithm}

Prop.~\ref{prop:pb3} suggests that minimizers of Problem~(\ref{eq:prox})
can be found by minimizing the right-hand side of~(\ref{eq:pb3}).
This can be achieved by means of the stochastic approximation algorithm
provided in Sec.~\ref{sec:as}.
The corresponding iterations~(\ref{eq:spa}) read as $x_{n+1}=\sT_{\gamma_{n+1}}(x_n,\xi_{n+1})$
where $(\xi_n)$ are iid copies of $\xi$.
For every $i=1,\ldots,L-1$, the intermediate variable $\bar x_{n+1}^i$ given by
Eq.~(\ref{eq:xintermediate}) satisfies
\begin{align*}
  \bar x_{n+1}^i &= \prox_{\gamma_n g_i(\,.\,,\xi_{n+1}^i)}(\bar x_n^{i-1}-\gamma_n \nabla f_i(\bar x_n^{i-1},\xi_{n+1}^i))\,.
\end{align*}
\begin{theorem}
Let Ass.~\ref{step} hold true. Assume that the convex function $F$ is 
differentiable and that $\nabla F$ is Lipschitz continuous. 
Assume that Problem~\eqref{eq:prox} admits a minimizer. Then, there exists a 
r.v.~$X_\star$ s.t. $X_\star(\omega)$ is a minimizer of~\eqref{eq:prox} 
for all $\omega$ $\bP$-a.e.,  and s.t. the sequence $(x_n)$ defined above converges a.s. to $X_\star$ as $n\to \infty$.
Moreover, for every $i=0,\ldots,L-1$, $\bar x_n^i$ converges a.s. to $X_\star$.
\end{theorem}
\begin{proof}
It is sufficient to verify that the mappings $f_i$, $g_i$ defined by (\ref{eq:fi}) and (\ref{eq:gi}) respectively
fulfill Ass.~\ref{fg}--\ref{bnd-mom} of Th.~\ref{th:main}. Then, Th.~\ref{th:main} gives the conclusion. Ass.~\ref{fg} and \ref{f-Lip} are trivially satisfied.
It remains to show, for every minimizer $x_\star$, the existence of a $(2+\varepsilon)$-representation, for some $\varepsilon>0$.
Any such $x_\star$ satisfies Eq. (\ref{eq:optimality-2}) where $\varphi_i$ satisfies~(\ref{eq:representation}).
By definition of $f_i$ and $g_i$, it is straightforward to show that there exists a deterministic constant $C_\star$ depending only 
on $x_\star$ and the graph $G$, such that $\|\nabla f_i(x_\star,\xi^i)\|<C_\star$ and $\|\varphi_i(\xi^i)\|<C_\star$.
This proves Ass.~\ref{selec}. Ass.~\ref{bnd-mom} can be easily checked by the same arguments.
\end{proof}

\begin{table}
\caption{Proposed Snake algorithm.}\label{tab:snake}
\begin{algorithmic}[t]
\Procedure{Snake}{$x_0, L$} 
\State $z\gets x_{0}$
\State $e\gets ${\sc Rnd\_oriented\_edge} 
\State $n\gets 0$
\State $\ell\gets L$
\While {stopping criterion is not met}
\State $c, e \gets $ {\sc{Simple\_path}}{$(e,\ell)$}
\State $z \gets $ {\sc{Prox1D}}($z - \gamma_n \frac{\text{\sc Length}(c)}{L|E|}\nabla F(z),c, \frac{1}{L}\gamma_n$)
\State $\ell\gets \ell-\text{\sc Length}(c)$
\If{$\ell = 0$}
\State $e \gets$ {\sc Rnd\_oriented\_edge}
\State $\ell \gets L$
\State $n \gets n+1$ \Comment{$x_n$is $z$ at this step}
\EndIf
\EndWhile
\State \textbf{return} $z$
\EndProcedure
\end{algorithmic}
\end{table}

\begin{table}
\caption{{\sc{Simple\_path}} procedure.}
\label{tab:path}
\begin{algorithmic}[t]
\Procedure{Simple\_path}{$e,\ell$} 
\State $c\gets e$
\State $w \gets $ {\sc Uniform\_Neib}$(e[-1])$
\While {$w \notin c$ and {\sc Length}(c)$<\ell$}
\State $c\gets [c,w]$
\State $w \gets $ {\sc Uniform\_Neib}$(w)$
\EndWhile
\State \textbf{return} $c, [c[-1],w]$
\EndProcedure
\end{algorithmic}
\end{table}

Consider the general $\phi$-regularized problem~\eqref{eq:prox}, and assume that an efficient procedure to compute the proximity operator of the $\phi$-regularization over an 1D-graph is available. The sequence $(x_n)$ is generated by the algorithm 
{\sc Snake} (applied with the latter 1D efficient procedure) and is summarized in Table~\ref{tab:snake}. Recall the definition of the probability $\pi$ on $V$ and the transition kernel $P$ on $V^2.$
The procedure presented in this table calls the following subroutines. 
\begin{itemize}
\item If $c$ is a finite walk, $c[-1]$ is the last element of $c$ and {\sc Length}$(c)$ is its length as a walk that is $|c|-1.$ 
\item The procedure {\sc Rnd\_Oriented\_Edge} returns a tuple of two nodes randomly chosen $(v,w)$ where $v \sim \pi$ and $w \sim P(v,.).$ 
\item For every $x\in \bR^V$, every simple path $s$ and every
  $\alpha>0$, {\sc Prox1D}$(x,s,\alpha)$ is any procedure that returns the
  quantity 
$
\prox_{\alpha R(.,\phi_s)}(x)\,.
$
\item The procedure {\sc Uniform\_Neib}$(v)$ returns a random
vertex drawn uniformly amongst the neighbors of the vertex $v$ that is with distribution $P(v,.)$.
\item The procedure {\sc Simple\_path}$(e,\ell)$, described
in Table~\ref{tab:path}, generates the first steps of a
random walk on $G$ with transition kernel $P$ 
initialized at the vertex $e[-1]$, and prefaced by the first node in $e$. 
It represents the $\xi^i$'s of the previous section. The random walk is stopped
when one node is repeated, or until the maximum number of samples
$\ell+1$ is reached. The procedure produces two outputs, the walk and the oriented edge $c, (c[-1],w)$. 
In the case where the procedure stopped due to a repeated node,
$c$ represents the simple path obtained by stopping the walk before
the first repetition occurs, while $w$ is the vertex which has been
repeated (referred to as the pivot node). In the case where no vertex
is repeated, it means that the procedure stopped because the maximum
length was achieved. In that case, $c$ represents the last simple
path generated, and the algorithm doesn't use the pivot node $w$.
\end{itemize}

\begin{remark}
  Although Snake converges for every value of the hyperparameter $L$, a natural question is about
the influence of $L$ on the behavior of the algorithm. In the case 
where $R(\,\cdot\,,\phi)$ is the TV regularization, \cite{condat2013direct} notes that, empirically,
the taut-string algorithm used to compute the proximity operator has a complexity of order $O(L)$.
The same holds for the Laplacian regularization. Hence, parameter $L$ controls the complexity of every iteration.
On the other hand, in the reformulation of Problem~(\ref{eq:prox}) into the stochastic form~(\ref{eq:tvE}),
the random variable $|E| R(x,\phi_\xi)/L$ is an unbiased estimate of $R(x,\phi)$. By the ergodic theorem, the larger $L$, 
the more accurate is the approximation. Hence, there is a
trade-off between complexity of an iteration and precision of the
algorithm.
This trade-off is standard in the machine learning literature. It
often appears while sampling mini-batches in order to apply the
stochastic gradient algorithm to a deterministic optimization
problem(see~\cite{bottou2010large, bottou2016optimization}). The choice of $L$ is somehow similar to
the problem of the choice of the length of the mini-batches in this context.

Providing a theoretical rule that would optimally select the value of $L$ is a difficult task
that is beyond the scope of this paper. Nevertheless, in Sec.~\ref{sec:examples}, 
we provide a detailed analysis of the influence of $L$ on the numerical performance of the algorithm.
\end{remark}

\section{Proximity operator over 1D-graphs}
\label{sec:1D}

We now provide some special cases of $\phi$-regularizations, for which the computation of the proximity operator over 1D-graphs is easily tractable. 
Specifically, we address the case of the total variation regularization and the Laplacian regularization which are particular cases of $\phi$-regularizations.

\subsection{Total Variation norm}

In the case where $\phi_{\{i,j\}}(x,x') = w_{\{i,j\}}|x - x'|,$ $R(x,\phi)$ reduces
to the weighted TV regularization $$ R(x,\phi) =
\sum_{\{i,j\}\in E} w_{\{i,j\}} |x(i)-x(j)| \,$$ and in the case where
$\phi_{\{i,j\}}(x,x') = |x - x'|,$ $R(x,\phi)$ reduces to the its
unweighted version $$ R(x,\phi) = \sum_{\{i,j\}\in E} |x(i)-x(j)|\,.$$

As mentioned above, there exists a fast method, the taut string algorithm, to compute the proximity operator of these $\phi-$regularizations over a 1D-graph (\cite{barberoTV14,condat2013direct}). 

\subsection{Laplacian regularization}

In the case where $\phi_{\{i,j\}}(x, x') = w_{\{i,j\}}(x - x')^2,$ $R(x,\phi)$ reduces to the Laplacian regularization that is $$ R(x,\phi) = \sum_{\{i,j\}\in E} w_{\{i,j\}} (x(i)-x(j))^2.$$ Its unweighted version is $$ \sum_{\{i,j\}\in E} (x(i)-x(j))^2 = \|\nabla x\|^2 = x^{*} \mathcal{L} x. $$

In the case where $\phi_{\{i,j\}}(x, x') = w_{\{i,j\}}(x/\sqrt{\deg(i)} - x'/\sqrt{\deg(j)})^2,$ $$R(x,\phi) = \sum_{\{i,j\}\in E} w_{\{i,j\}} \left (\frac{x(i)}{\sqrt{\deg(i)}} - \frac{x'(i)}{\sqrt{\deg(j)}} \right )^2$$ is the normalized Laplacian regularization.

We now explain one method to compute the proximity operator of the unweighted Laplacian regularization over an 1D-graph. The computation of the proximity operator of the normalized Laplacian regularization can be done similarly. The computation of the proximity operator of the weighted Laplacian regularization over an 1D-graph is as fast as the computation the proximity operator of the unweighted Laplacian regularization over an 1D-graph, using for example Thomas' algorithm.

The proximity operator of a fixed point $y\in \bR^{\ell+1}$ is obtained as a solution to a quadratic programming problem of the form:
$$\min_{x \in \bR^{\ell+1}} \frac12 \|x - y\|^2 + \lambda\sum_{k = 1}^{\ell} (x(k-1) - x(k))^2\,,$$
where $\lambda > 0$ is a scaling parameter.
Writing the first order conditions, the solution $x$ satisfies
\begin{equation} 
(I +  2\lambda \mathcal{L}) x = y
\label{eq:systemelin} 
\end{equation} 
where $\mathcal{L}$ is the Laplacian matrix of the 1D-graph with $\ell+1$ nodes and $I$ is the identity matrix in $\bR^{\ell+1}$. 
By \cite{graham1997spectral}, $\mathcal{L}$ can be diagonalized explicitely. In particular, $I +  2\lambda \mathcal{L}$ has eigenvalues $$1 + 4\lambda \left(1 - \cos \left( \frac{\pi k}{\ell+1} \right) \right),$$ and eigenvectors $e_k \in \bR^{\ell+1}$ $$e_{k}(j) = \frac{1}{2(\ell+1)}\cos \left (\pi \frac{kj}{\ell+1} - \pi\frac{k}{2(\ell+1)} \right),$$ for $0 \leq k < n$.
Hence, $x = C^{*} \Lambda^{-1} C y$, where $\Lambda$ gathers the eigenvalues of $I +  2\lambda \mathcal{L}$ and the operators  $C$ and $C^{*}$ are the discrete cosine transform operator and the inverse discrete cosine transform.
The practical computation of $x$ can be found in $O(\ell\log(\ell))$ operations.

\section{Examples}
\label{sec:examples}

We now give some practical instances of Problem~\eqref{eq:prox} by particularizing $F$ and the $\phi$-regularization in~\eqref{eq:prox}. We also provide some simulations to compare our method to existing algorithms.  

\subsection{Trend Filtering on Graphs}

Consider a vector $y \in \bR^V$. The Graph Trend Filtering (GTF) estimate on 
$V$ with parameter $k$ set to one is defined in \cite{wang2014trend} by 
\begin{equation}
\hat{y} = \arg\min_{x\in \bR^V} \frac 12\|x-y\|^2 + \lambda \sum_{\{i,j\} \in E} |x(i) - x(j)|.
\label{eq:proxtv}
\end{equation}
where $\lambda > 0$ is a scaling parameter.  
In the GTF context, the vector $y$ represents a sample of noisy data over the graph $G$ and the GTF estimate represents a denoised version of $y$. When $G$ is an 1D or a 2D-graph, the GTF boils down to a well known context~\cite{tibshirani2014adaptive, chambolle2010introduction}. When $G$ is a general graph, the GTF estimate is studied in \cite{wang2014trend} and \cite{hutter2016optimal}. 
The estimate $\hat y$ is obtained as the solution of a TV-regularized risk minimization with $F(x) = \frac12 \|x-y\|^2$ where $y$ is fixed. We address the problem of computing the GTF estimate on two real life graphs from \cite{snapnets} and one sampled graph. 
The first one is the Facebook graph which is a network of 4039 nodes and 88234 edges extracted from the Facebook social network. The second one is the Orkut graph with 3072441 nodes and 117185083 edges. Orkut was also an on-line social network. \as{The third graph is sampled according to a Stochastic Block Model (SBM). Namely we generate a graph of 4000 nodes with four well-separated clusters of 1000 nodes (also called ``communities'') as depicted in Fig.~\eqref{fig:data}. Then we
draw independently $N^2$ Bernoulli r.v. $E(i,j)$, encoding the edges of the
graph (an edge between nodes $i$ and $j$ is present iff $E(i,j)=1$), such that
$\mathbb{P}\{E(i,j) = 1\} = P(c_i,c_j)$ where $c_i$ denotes the community of the node $i$ and where

\[
\begin{cases}
P(c,c') =& .1\text{ if } c=c'\\
P(c,c') =& .005\text{ otherwise}
\end{cases}
\]
This model is called the stochastic block model for the matrix
$P$~\cite{holland1983stochastic}. It amounts to a blockwise Erd\"os-R\'enyi
model with parameters depending only on the blocks. It leads to 81117 edges.

We assume that every node is provided with an unknown value in $\bR$ 
(the set of all these values being referred to as the \emph{signal} in the sequel).
In our example, the value $y(i)$ at node $i$ is generated as $y(i) = l(c_i) + \sigma \epsilon_i$ where $l$ is a 
mapping from the communities to a set of levels (in Fig.~\ref{fig:data}, $l(i)$ is an integer in $[0,255]$), 
and $\epsilon$ denotes a standard Gaussian white noise with $\sigma>0$ as its standard deviation. In Figure~\ref{fig:data} we
represent an example of the signal $y$ (left figure) along with the ``initial'' values $l(c_i)$ represented in grayscale at every node.


\begin{figure}[ht!]
\[
  \begin{array}{cc}
  \includegraphics[width=.5\linewidth]{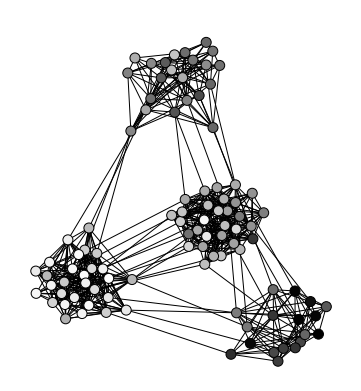} &
  \includegraphics[width=.5\linewidth]{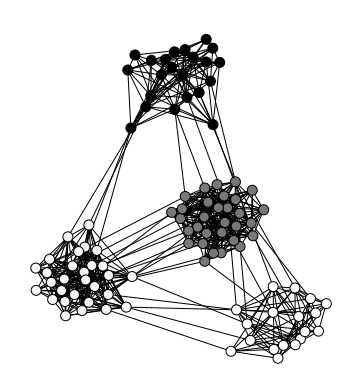}
  \end{array}
\]
\caption{The signal is the grayscale of the node The graph is sampled according to a SBM. Left: Noised signal over the nodes. Right: Sought signal.}
\label{fig:data}
\end{figure}

}

Over the two real life graphs, the vector $y$ is sampled according to a standard Gaussian distribution of dimension $|V|$. The parameter $\lambda$ is set such that $\bE[\frac12 \|x-y\|^2] = \bE[\lambda \sum_{\{i,j\} \in E} |x(i) - x(j)|]$ if $x,y$ are two independent r.v with standardized Gaussian distribution. The initial guess $x_0$ is set equal to $y$. The step size $\gamma_n$ set equal to $|V|/(10 n)$ for the two real life graphs and $|V|/(5 n)$ for the SBM realization graph. We ran the Snake algorithm for different values of $L$, except over the Orkut graph where $L = |V|$.

The dual problem of~\eqref{eq:proxtv} is quadratic with a box
constraint. The Snake algorithm is compared to the well-known
projected gradient (PG) algorithm for the dual problem. To solve the dual problem of~\eqref{eq:proxtv}, we use L-BFGS-B~\cite{l-bfgs-b(95)} as suggested in~\cite{wang2014trend}. 
Note that, while running on the
Orkut graph, the algorithm L-BFGS-B leads to a memory error from the solver~\cite{l-bfgs-b(95)} in SciPy (using one thread of a 2800 MHz CPU and 256GB RAM). 

Figures~\ref{fig:Courbe},~\ref{fig:Courbe1} and~\ref{fig:Courbe2} show the objective function as a function of time for each algorithm.

\begin{figure}[ht!]
\includegraphics[width=\linewidth]{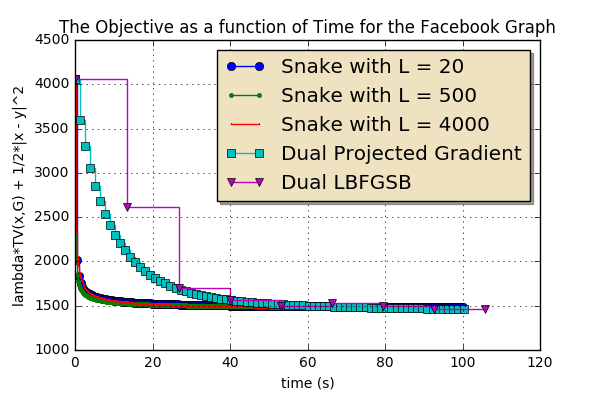}
\caption{Snake applied to the TV regularization over the Facebook Graph}
\label{fig:Courbe}
\end{figure}

\begin{figure}[ht!]
\includegraphics[width=\linewidth]{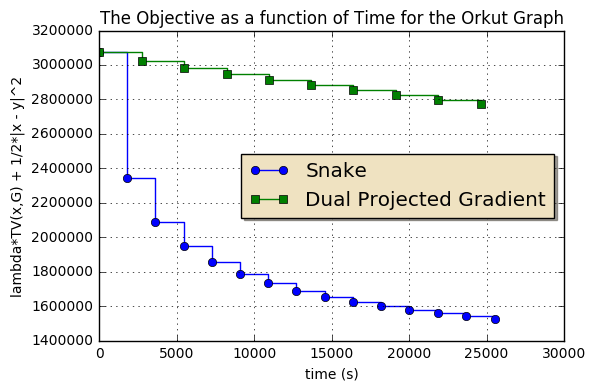}
\caption{Snake applied to the TV regularization over the Orkut Graph}
\label{fig:Courbe1}
\end{figure}

\begin{figure}[ht!]
\includegraphics[width=\linewidth]{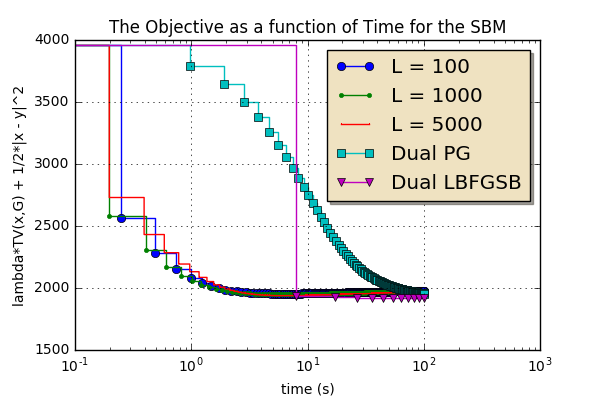}
\caption{Snake applied to the TV regularization over the SBM realization graph, in log scale}
\label{fig:Courbe2}
\end{figure}

In the case of the TV regularization, we observe that Snake takes advantage of being an online method, which is known to be twofold (~\cite{bottou2010large, bottou2016optimization}). First, the iteration complexity is controlled even over large general graphs: the complexity of the computation of the proximity operator is empirically linear~\cite{condat2013direct}. On the contrary, the projected gradient algorithm involves a matrix-vector product with complexity $O(|E|)$. Hence, \textit{e.g} the projected gradient algorithm has an iteration complexity of at least $O(|E|)$. The iteration complexity of Snake can be set to be moderate in order to frequently get iterates while running the algorithm. Then, Snake is faster than L-BFGS-B and the projected gradient algorithms for the dual problem in the first iterations of the algorithms. 

Moreover, for the TV regularization, Snake seems perform globally better than L-BFGS-B and the projected gradient. This is because Snake is a proximal method where the proximity operator is efficiently computed (\cite{bau-com-livre11}).

The parameter $L$ seems to have a minor influence on the performance of the algorithm since, in Figure~\ref{fig:Courbe} the curves corresponding to different values of $L$ are closely superposed. The log scale in Figure~\ref{fig:Courbe2} allows us to see that the curve corresponding Snake with $L = 1000$ performs slightly better that the others. Figure~\ref{fig:Courbe3} shows supplementary curves in log scale where Snake is run over the Facebook graph with different values of $L$.

\begin{figure}[ht!]
\includegraphics[width=\linewidth]{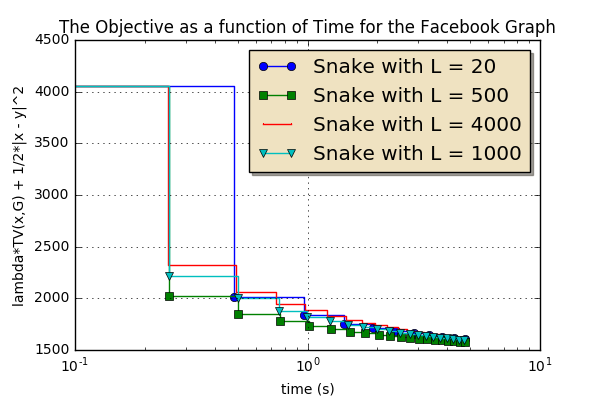}
\caption{Snake applied to the TV regularization over the Facebook graph, in log scale}
\label{fig:Courbe3}
\end{figure}

In Figure~\ref{fig:Courbe3}, the best performing value is $L = 500$.

Over the three graphs, the value $L = O(|V|)$ is a good value, if not the best value to use the Snake algorithm. One can show that, while sampling the first steps of the infinite random walk over $G$ from the node, say $v$, the expected time of return to the random node $v$ is $|V|$. Hence, the value $L = |V|$ allow Snake to significantly explore the graph during one iteration. 

\subsection{Graph Inpainting}

The problem of graph inpainting has been studied in
\cite{chen2014signal,el2016asymptotic,zhu2003semi} and can be
expressed as follows. Consider a vector $y \in \bR^V$, a subset $O
\subset V$. Let $\bar O$ be its complementary in $V$. The harmonic energy
minimization problem is defined in~\cite{zhu2003semi} by
\begin{equation*}
\begin{aligned}
& \underset{x \in \bR^{V}}{\text{min}}
& & \sum_{\{i,j\} \in E} (x(i) - x(j))^2 \\
& \text{subject to}
& & x(i) = y(i), \forall i \in O.
\end{aligned}
\end{equation*}
This problem is interpreted as follows. The signal $y \in \bR^V$ is partially observed over the nodes and the aim is to recover $y$ over the non observed nodes. The subset $O \subset V$ is the set of the observed nodes and $\bar O$ the set of unobserved nodes. An example is shown in Figure~\ref{fig:inpainting}.
\begin{figure}[ht!]
\[
  \begin{array}{cc}
  \includegraphics[width=.5\linewidth]{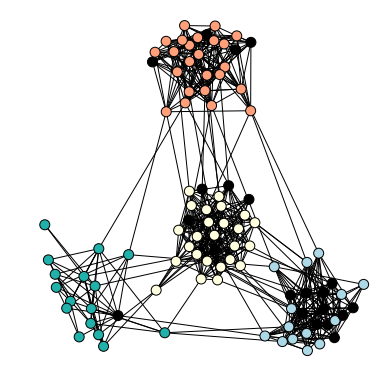} &
  \includegraphics[width=.5\linewidth]{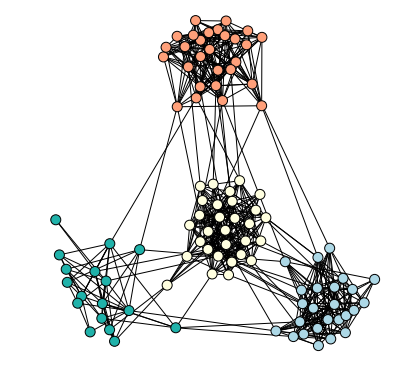}
  \end{array}
\]
\caption{Left: Partially observed data (unobserved nodes are black, data is the color of nodes). Right: Fully observed data (the color is observed for all nodes).}
\label{fig:inpainting}
\end{figure}

Denote by $G_{\bar O} = (\bar O, E_{\bar O})$ the subgraph of $G$
induced by $\bar O$. Namely, $\bar O$ is the set of vertices, and the
set $E_{\bar O}$ is formed by the edges $\{i,j\}\in E$ s.t. $i\in \bar
O$ and $j\in \bar O$.  The harmonic energy minimization is equivalent
to the following Laplacian regularized problem over the graph $G_{\bar O}$:
\begin{equation}
\label{eq:harmonic-energy-minimization}
\begin{aligned}
& \underset{x \in \bR^{\bar O}}{\text{min}}
& & F(x)  + \sum_{\{i,j\}\in E_{\bar O} \atop i < j} (x(i) - x(j))^2 
\end{aligned}
\end{equation}
where 
$$
F(x) = \sum_{i \in \bar O, j \in O \atop \{i,j\} \in E} (x(i) - y(j))^2\,.
$$
The signal $y$ is sampled according to a standardized Gaussian distribution of dimension $|V|$. 
We compared the Snake algorithm to existing algorithm over the Orkut graph. The set $V$ is divided in two parts of equal size to define $O$ and $\bar O$. The initial guess is set equal to zero over the set of unobserved nodes $\bar O$, and to the restriction of $y$ to $O$ over the set of observed nodes $O$. 
We compare our algorithm with the conjugate gradient


Figures~\ref{fig:laplace} and \ref{fig:laplace2} represent the
objective function $\sum_{\{i,j\} \in E} (x(i) - x(j))^2$ as a
function of time. Over the Facebook graph, the parameter $L$ is set equal to
$|V|/10$. The step size $\gamma_n$ are set equal to $|V|/(10 n)$. Over the
Orkut graph, $L$ is set equal to $|V|/50$. The step size are set equal
to $|V|/(5 \sqrt{n})$ on the range displayed in Figure~\ref{fig:laplace}. Even if the sequence $(|V|/(5 \sqrt{n}))_{n \in \bN}$ does not satisfies the Ass.~\ref{step}, it is a standard trick in stochastic approximation to take a slowly decreasing step size in the first iterations of the algorithm (\cite{moulines2011non}). It allows the iterates to be quickly close to the set of solutions without converging to the set of solutions. Then, one can continue the iterations using a step size satisfying Ass.~\ref{step} to make the algorithm converging. There is a trade-off between speed and precision while choosing the step-size.
\begin{figure}[ht!]
\includegraphics[width=\linewidth]{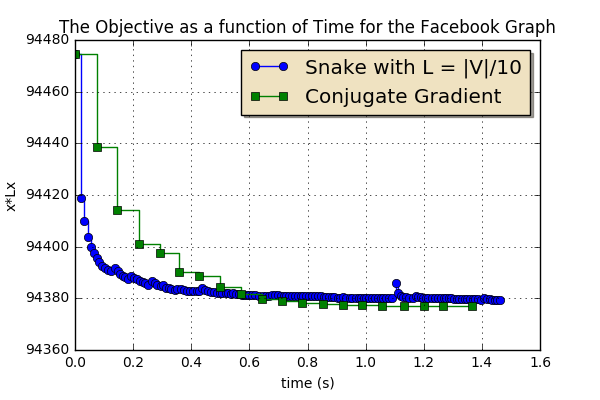}
\caption{Snake applied to the Laplacian regularization over the Facebook Graph}
\label{fig:laplace}
\end{figure}
\begin{figure}[ht!]
\includegraphics[width=\linewidth]{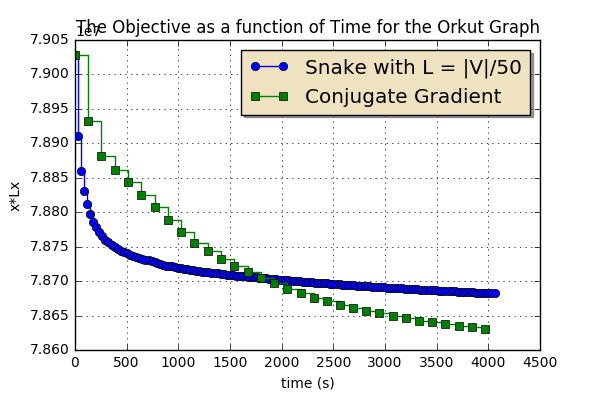}
\caption{Snake applied to the Laplacian regularization over the Orkut Graph}
\label{fig:laplace2}
\end{figure}
Snake turns out to be faster in the
first iterations. Moreover, as an online method, it
allows the user to control the iteration complexity of the
algorithm. Since a discrete cosine transform is used, the complexity of the computation of the proximity operator is $O(L \log(L))$. In contrast, the iteration complexity of the conjugate
gradient algorithm can be a bottleneck (at least $O(|E|)$) as far as very large graphs are
concerned.

Besides, Snake for the Laplacian regularization does not perform globally better than the conjugate gradient. This is because the conjugate gradient is designed to fully take advantage of the quadratic structure. On the contrary, Snake is not specific to quadratic problems.  


\subsection{Online Laplacian solver}

Let $\cL$ the Laplacian of a graph $G = (V,E)$.  The resolution of the
equation $\cL x = b$, where $b$ is a zero mean vector, has numerous
applications (\cite{vishnoi2012laplacian,spielman2010algorithms}). It
can be found by minimizing
the Laplacian regularized problem $$\min_{x \in \bR^V} -b^{*}x +
\frac12 x^{*} \cL x.$$ In our experiment, the vector $b$ is sampled
according to a standardized Gaussian distribution of dimension $|V|$.
We compare our algorithm with the conjugate gradient over the Orkut
graph.

\begin{figure}[ht!]
\includegraphics[width=\linewidth]{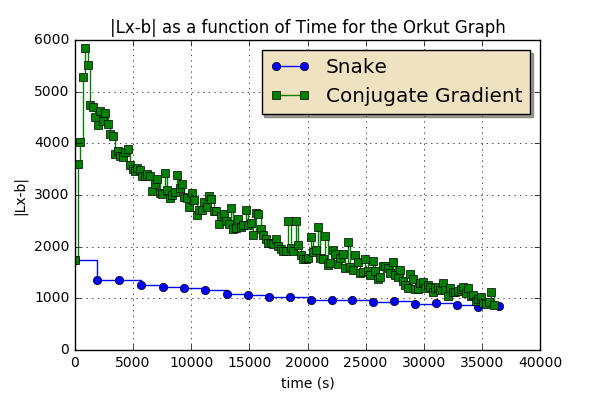}
\caption{Snake applied to the resolution of a Laplacian system over the Orkut graph}
\label{fig:Lxb}
\end{figure}

Figure~\ref{fig:Lxb} represents the quantity $\|\mathcal L x_n - b\|$
as a function of time, where $x_n$ is the iterate provided either by
Snake or by the conjugate gradient method. The parameter $L$ is set
equal to $|V|$. The step size $\gamma_n$ are set equal to $|V|/(2 n)$.
Snake appears to be more stable than the conjugate gradient method, 
has a better performance at start up.

\section{Proof of Th.~\ref{th:main}}
\label{sec:prf-main}

We start with some notations. We endow the probability space 
$(\Omega, \mcF, \bP)$ with the filtration $(\mcF_n)$ defined as 
$\mcF_n = \sigma( \xi_1, \ldots, \xi_n)$, and we write 
$\bE_n = \bE[\cdot\, | \, \mcF_n]$. In particular, $\bE_0 = \bE$.  We also
define $G_i(x) = \bE[g_i(x,\xi^i)]$ and $F_i(x) = \bE[f_i(x,\xi^i)]$ for every
$x\in \sX$.  We denote by $\mu_i$ and $\mu$ the probability laws of $\xi^i$ and
$\xi$ respectively. Finally, $C$ and $\eta$ will refer to positive constants 
whose values can change from an equation to another. The constant $\eta$ can 
be chosen arbitrarily small.

In \cite{bia-hac-16}, the case $L=1$ is studied (algorithm~\eqref{eq:sfb}).
Here we shall reproduce the main steps of the approach of \cite{bia-hac-16},
only treating in detail the specificities of the case $L \geq 1$.  We also note
that in \cite{bia-hac-16}, we considered the so-called maximal monotone
operators, which generalize the subdifferentials of convex functions.
This formalism is not needed here.

The principle of the proof is the following. Given $a \in \sX$, consider the
so called differential inclusion (DI) defined on the set of absolutely 
continuous functions from $\bR_+ = [0,\infty)$ to $\sX$ as follows: 
\begin{equation} 
\label{eq:di-intro}
\left\{\begin{array}[h]{lcl}
\dot \sz(t) &\in& 
    - \sum_{i=1}^L (\nabla F_i(\sz(t)) + \partial G_i(\sz(t)))\, \\ 
 \sz(0) &=& a \, . 
\end{array}\right. 
\end{equation}
It is well known that this DI has a unique solution, \emph{i.e.}, a unique
absolutely continuous mapping $\sz : \bR_+ \to \sX$ such that $\sz(0) = a$, and
$\dot \sz(t) \in - \sum (\nabla F_i(\sz(t)) + \partial G_i(\sz(t)))$ for
almost all $t > 0$. Consider now the map $\Phi : \sX \times \bR_+ \to \sX$,
$(a,t) \mapsto \sz(t)$, where $\sz(t)$ is the DI solution with the initial 
value $\sz(0) = a$. Then, $\Phi$ is a semi-flow \cite{bre-livre73, aub-cel-(livre)84}. 

Let us introduce the following function $\sI$ from $\sX^{\bN}$ to the space of 
$\bR_+ \to \sX$ continuous functions. For $u = (u_n) \in \sX^{\bN}$, the 
function $\su = \sI(u)$ is the continuous interpolated process obtained from 
$u$ as 
\[
\su(t) = u_n + \frac{u_{n+1} - u_n}{\gamma_{n+1}} (t - \tau_n)
\]
for $t \in [\tau_n, \tau_{n+1})$, where $\tau_n = \sum_{1}^n \gamma_k$. 
Consider the interpolated function $\sx = \sI((x_n))$. We shall prove
the two following facts: 
\begin{itemize} 
\item The sequence $(\| x_n - x_\star \|)$ is almost surely convergent for
each $x_\star \in \mZ$ (Prop.~\ref{fejer}); 
\item The process $\sx(t)$ is an almost sure Asymptotic Pseudo Trajectory (APT)
of the semi-flow $\Phi$, a concept introduced by Bena\"\i m and Hirsch in the 
field of dynamical systems \cite{ben-hir-96}. Namely, for each $T > 0$, 
\begin{equation}
\label{as-apt}
\sup_{u\in[0,T]} \| \sx(t+u) - \Phi(\sx(t), u) \| 
\xrightarrow[t\to\infty]{\text{a.s.}} 0 , 
\end{equation} 
\end{itemize}
Taken together, these two results lead to the a.s.~convergence of $(x_n)$ to
some r.v. $X^\star$ supported by the set $\mZ$, 
as is shown by \cite[Cor.~3.2]{bia-hac-16}. The convergence of the
$(\bar x_n^i)_n$ stated by Th.~\ref{th:main} will be shown in the course of
the proof. 

Before entering the proof, we recall some well known facts relative to the
so called Moreau envelopes. For more details, the reader is referred to
\emph{e.g.} \cite[Ch.~2]{bre-livre73}, or \cite[Ch.~12]{bau-com-livre11}. 
The Moreau envelope of parameter $\gamma$ of a convex function $h$ with domain
$\sX$ is the function 
\[
h^\gamma(x) = \min_{w\in\sX} h(w) + (2\gamma)^{-1} \| w - x \|^2 \, . 
\]
The function $h^\gamma$ is a differentiable function on $\sX$, and its 
gradient is given by the equation 
\begin{equation}
\label{yosida} 
\nabla h^\gamma(x) = \gamma^{-1}( x - \prox_{\gamma h}(x)) \, . 
\end{equation} 
This gradient is a $\gamma^{-1}$-Lipschitz continuous function satisfying 
the inequality $\| \nabla h^\gamma(x) \| \leq \| \partial h^0(x) \|$, where
$\partial h^0(x)$ is the least-norm element of $\partial h(x)$. Finally, for 
all $(x,u) \in \sX \times \sX$ and for all $v\in\partial h(u)$, the inequality 
\begin{equation}
\label{yos-mon} 
\ps{\nabla h^\gamma(x) - v, \prox_{\gamma h}(x) - u} \geq 0 
\end{equation} 
holds true. 
With the formalism of the Moreau envelopes, the mapping $\sT_{\gamma,i}$ 
can be rewritten as 
\begin{equation}
\label{Tyosida} 
\sT_{\gamma,i}(x, s) = x - \gamma \nabla f_i(x,s) - 
\gamma\nabla g^\gamma_i(x - \gamma \nabla f_i(x,s), s) 
\end{equation} 
thanks to \eqref{yosida}, where $\nabla g^\gamma_i(\cdot, s)$ is the
gradient of the Moreau envelope $g_{i}^\gamma(\cdot, s)$. We shall adopt this form 
in the remainder of the proof. 

The following lemma is proven in Appendix~\ref{prf-nf-ng}. 
\begin{lemma}
\label{bnd-nf-ng} 
For $i=1,\ldots, L$, let 
\[
\bar x^i = (T_{\gamma,i}(\cdot,s^i) \circ \cdots \circ 
T_{\gamma,1}(\cdot,s^1))(x) . 
\]
Then, with Ass.~\ref{f-Lip}, there exists a measurable map $\kappa:\Xi^L\to \bR_+$
s.t. $\bE[\kappa(\xi)^\alpha]<\infty$ for all $\alpha\geq 1$ and s.t. for all $\bar s = (s^1,\ldots,s^L) \in \Xi^L$, 
\begin{align*} 
& \| \nabla f_i( \bar x^{i-1}, s^i ) \| 
\leq 
\kappa(\bar s) \sum_{k=1}^i \| \nabla f_k(x,s^k) \| + \| \nabla g_k^\gamma(x,s^k) \| \\
&\| \nabla g_i^\gamma(\bar x^{i-1} 
   - \gamma \nabla f_i( \bar x^{i-1}, s^i ), s^i) \| \\
& \ \ \ \ \ \ \ \ \ \ \ \ \ \ \ \ \ 
\leq \kappa(\bar s) \sum_{k=1}^i \| \nabla f_k(x,s^k) \| + \| \nabla g_k^\gamma(x,s^k) \| .  
\end{align*} 
\end{lemma} 

Recall that we are studying the iterations 
$\bar x_{n+1}^i = 
\sT_{\gamma_{n+1},i}( \bar x_{n+1}^{i-1}, \xi_{n+1}^i)$, for 
$i = 1,\ldots, L$, $n \in \bN^{*}$, with $\bar x_{n+1}^0 = x_n$ and 
$x_{n+1} = \bar x_{n+1}^L$. In this section and in Appendix~\ref{anx-prf-main}, 
we shall write for conciseness, for any $x_\star\in\mZ$,
\begin{align*}
\nabla g^\gamma_i &= 
 \nabla g_i^{\gamma_{n+1}}(\bar x_{n+1}^{i-1} 
- \gamma_{n+1} \nabla f_i(\bar x_{n+1}^{i-1}, \xi_{n+1}^i), \xi_{n+1}^i), \\ 
\prox_{\gamma g_i} &= 
\prox_{\gamma g_i(\cdot,\xi_{n+1}^i)}(
\bar x_{n+1}^{i-1} 
   - \gamma_{n+1} \nabla f_i(\bar x_{n+1}^{i-1}, \xi_{n+1}^i) ) , \\ 
\nabla f_i &= \nabla f_i(\bar x_{n+1}^{i-1}, \xi_{n+1}^i), \\ 
\nabla f_i^\star &= \nabla f_i(x_\star,\xi_{n+1}^i) \ \text{where } 
    x_\star \in \mZ, \\
\varphi_i &= \varphi_i(\xi_{n+1}^i), \ \text{(see Ass.~\ref{selec})}  
\quad \text{and} \\ 
\gamma &= \gamma_{n+1}. 
\end{align*} 
The following proposition is analoguous to \cite[Prop.~1]{bia-16} 
or to \cite[Prop.~6.1]{bia-hac-16}: 
\begin{prop}
\label{fejer}
Let Ass.~\ref{step}--\ref{selec} hold true. Then the following facts hold true: 
\begin{enumerate}
\item\label{sto-fejer} For each $x_\star \in \mZ$, the sequence 
$(\| x_n - x_\star\|)$ converges almost surely.  
\item\label{mart-L2}
$\bE \left[ \sum_{i=1}^L \sum_{n=1}^\infty \gamma^2 
( \| \nabla g^\gamma_i \|^2 + \| \nabla f_i \|^2 ) \right]< \infty$. 
\item\label{xbar-x}
For each $i$, $\bar x_{n+1}^i - x_n \to 0$ almost surely. 
\end{enumerate} 
\end{prop} 
This proposition is shown in Appendix~\ref{anx-prf-fejer}.  
It remains to establish the almost sure APT to prove
Th.~\ref{th:main}. We just provide here the main arguments of this part of
the proof, since it is similar to its analogue in \cite{bia-hac-16}. 

Let us write 
\begin{multline*} 
x_{n+1} = x_n - \gamma_{n+1} 
\sum_{i=1}^L \Bigl(\nabla f_i(\bar x_{n+1}^{i-1}, \xi_{n+1}^i) \\ 
  + \nabla g_i^{\gamma_{n+1}}(\bar x_{n+1}^{i-1}
 - \gamma_{n+1} \nabla f_i(\bar x_{n+1}^{i-1}, \xi_{n+1}^i), \xi_{n+1}^i) 
  \Bigr)\, , 
\end{multline*} 
and let us also define the function 
\begin{align*} 
H_\gamma(x,(s^1,\ldots,s^L)) &= \sum_{i=1}^L 
\left[\nabla f_i(\bar x^{i-1}, s^i) \right. \\ 
&\phantom{=} \ \ \ \ \ \ \ \ 
 \left. + \nabla g_i^\gamma( \bar x^{i-1} 
   - \gamma \nabla f_i(\bar x^{i-1},s^i), s^i) \right]  \, , 
\end{align*} 
where  we recall the notation
$\bar x^i = (T_{\gamma,i}(\cdot,s^i) \circ \cdots \circ 
T_{\gamma,1}(\cdot,s^1))(x)$. 
By Lem.~\ref{bnd-nf-ng} and Ass.~\ref{f-Lip}, \ref{selec} and \ref{bnd-mom},
$\bE[\|H_\gamma(x,\xi)\|]<\infty$ and we define:
\[ 
h_\gamma(x) = \bE [H_\gamma(x,\xi)]\,.
\] 
Note that $x_{n+1} = x_n-\gamma_{n+1}H_{\gamma_{n+1}}(x_n,\xi_{n+1})$.  
Defining the $(\mcF_n)$ martingale 
\[
M_n = \sum_{k=1}^n x_k - \bE_{k-1}[x_k]
\] 
it is clear that $x_{n+1} = x_n - \gamma_{n+1} h_{\gamma_{n+1}}(x_n) + 
(M_{n+1} - M_n)$. Let us rewrite this equation in a form involving the 
continuous process $\sx = \sI((x_n))$. Defining $\sM = \sI((M_n))$, and 
writing 
\[
r(t) = \max \{ k\geq 0 \, : \, \tau_k \leq t \}, \quad t \geq 0, 
\]
we obtain 
\begin{align} 
\sx(\tau_n + t) - \sx(\tau_n) &= 
 - \int_0^t h_{\gamma_{r(\tau_n+u)+1}}(x_{r(\tau_n+u)}) \, du  \nonumber \\ 
&\phantom{=} + \sM(\tau_n+t) - \sM(\tau_n) \, . 
\label{eq-integ} 
\end{align} 

The first argument of the proof of the almost sure APT is a compactness
argument on the sequence of continuous processes $(\sx(\tau_n+\cdot))_n)$.
Specifically, we show that on a $\bP$-probability one set, this sequence is
equicontinuous and bounded. By Ascoli's theorem, this sequence admits
accumulation points in the topology of the uniform convergence on the compacts
of $\bR_+$. As a second step, we show that these accumulation points are
solutions to the differential inclusion \eqref{eq:di-intro}, which is in fact
a reformulation of the almost sure APT property \eqref{as-apt}. 

Since 
\begin{align*}
&\bE \left[ \| x_{n+1} - \bE_n x_{n+1} \|^2 \right] \\
&= \gamma^2 \bE \left[ \Bigl\| 
\sum_{i=1}^L (\nabla f_i - \bE_n \nabla f_i) + 
\sum_{i=1}^L (\nabla g_i^\gamma - \bE_n \nabla g_i^\gamma) \Bigr\|^2 \right]\\
&\leq C \gamma^2 \bE \left[ \sum_{i=1}^L 
  (\| \nabla f_i \|^2 + \| \nabla g_i^\gamma \|^2) \right] \, , 
\end{align*}
we obtain by Prop.~\ref{fejer}--\ref{mart-L2}) that 
$\sup_n \bE [\| M_n \|^2] < \infty$. Thus, the martingale $M_n$ converges almost
surely, which implies that the sequence 
$(\sM(\tau_n+\cdot) - \sM(\tau_n))_{n}$ converges almost surely to zero,  
uniformly on $\bR_+$. 

By Ass.~\ref{f-Lip} and~\ref{selec}, 
$\sup_{x \in \cK} \int \| \nabla f_i(x,s) \|^2 \mu_i(ds) < \infty$
for each compact $\cK \subset \sX$ and each $i$. By 
Ass.~\ref{bnd-mom}, we also have \begin{align*}
\sup_{x \in \cK} \int \| \nabla g_i^\gamma(x,s) \|^{1+\varepsilon} \mu_i(ds) 
&\leq 
\sup_{x \in \cK} \int \| \partial g_i^0(x,s) \|^{1+\varepsilon} \mu_i(ds) \\
&< \infty \, .
\end{align*} 
Thus by Lem.~\ref{bnd-nf-ng} and Hölder inequality, and using the fact that the
sequence  $(x_n)$ is almost surely bounded by Prop.~\ref{fejer}--\ref{sto-fejer}), 
it can be shown that
\[
\sup_n \| h_{\gamma_{n+1}}(x_n) \| < \infty \quad \text{w.p. } 1 \, , 
\]
Inspecting \eqref{eq-integ}, we thus obtain that
the sequence $(\sx(\tau_n+\cdot))_n$ is equicontinuous and bounded with 
probability one. 

In order to characterize its cluster points, choose $T > 0$, and consider an elementary
event on the probability one set where $\sx$ is equicontinuous and bounded on
$[0,T]$.  With a small notational abuse, let $(n)$ be a subsequence along which
$(\sx(\tau_n + \cdot))_n$ converges on $[0,T]$ to some continuous function
$\sz(t)$. This function then is written as 
\[
\sz(t) - \sz(0) \!=\! 
- \lim_{n\to\infty} \int_0^t \!du \int_{\Xi} \!\mu(ds) \, 
 H_{\gamma_{r(\tau_n+u)+1}}(x_{r(\tau_n+u)}, s) . 
\] 
By the boundedness of $(x_n)$ (Prop.~\ref{fejer}-\ref{sto-fejer})), 
Lem.~\ref{bnd-nf-ng}, and Ass.~\ref{f-Lip},~\ref{selec} and \ref{bnd-mom}, 
the sequence of functions 
$(H_{\gamma_{r(\tau_n+u)+1}}(x_{r(\tau_n+u)}, s))_n$ in the parameters
$(u,s)$ is bounded in the Banach space $\cL^{1+\varepsilon}(du\otimes \mu)$, for some $\varepsilon > 0,$ 
where $du$ is the Lebesgue measure on $[0,T]$. 
Since the unit ball of $\cL^{1+\varepsilon}(du\otimes \mu)$ is weakly compact
in this space by the Banach-Alaoglu theorem, since this space is reflexive, we can extract a subsequence (still denoted as $(n)$) such that
$H_{\gamma_{r(\tau_n+u)+1}}(x_{r(\tau_n+u)}, s)$ converges weakly in
$\cL^{1+\varepsilon}(du\otimes \mu)$, as $n\to\infty$, to a function 
$Q(u,s)$.
The remainder of the proof consists in showing that $Q$ can be 
written as 
\[
Q(u,s) = \sum_{i=1}^L \left( b_i(u, s^i) + p_i(u, s^i) \right) \, , 
\]
where $b_i(u, s^i) = \nabla f_i(\sz(u), s^i)$ and 
$p_i(u, s^i) \in \partial g_i(\sz(u), s^i)$ for $du\otimes \mu_i$-almost 
all $(u,s^i)$. Indeed, once this result is established, it becomes clear that
$\sz(t)$ is an absolutely continuous function whose derivative satisfies
almost everywhere the inclusion 
$\dot\sz(t) \in - \sum_i (\nabla F_i(\sz(t)) + \partial G_i(\sz(t)))$,
noting that we can exchange the integration and the differentiation 
(resp.~the subdifferentiation) in the expression of $\nabla F$ (resp.~of 
$\partial G$). 

We just provide here the main argument of this part of the proof, since it is 
similar to its analogue in \cite{bia-hac-16}. Let $i \in \{1,\ldots,L\}$. Let us focus on the sequence 
of functions of $(u,s) \in [0,T] \times \Xi$ defined by 
\[
\nabla g^{\gamma_{r(\tau_n+u)+1}}_i( 
\bar x^{i-1}_{r(\tau_n+u)+1} - \gamma_{r(\tau_n+u)+1} 
\nabla f_i( \bar x^{i-1}_{r(\tau_n+u)+1}, s), s)
\]
and indexed by $n$.
This sequence is bounded in $\cL^{1+\varepsilon}(du\otimes \mu_i)$ on a probability one set, as a function of $(u,s)$, for the same reasons as those explained above for $(H_{\gamma_{r(\tau_n+u)+1}}(x_{r(\tau_n+u)}, s))_n$. We need to
show that any weak limit point $p_i(u,s)$ of this sequence satisfies 
$p_i(u, s) \in \partial g_i(\sz(u), s)$ for $du\otimes \mu_i$-almost all 
$(u,s)$. Using the fact that $\sx(\tau_n+\cdot) \to \sz(\cdot)$ almost surely, 
along with the inequality 
$\ps{\nabla g_i^\gamma(x,s) - w, \prox_{\gamma g_i(\cdot,s)}(x) - v} 
\geq 0$, valid for all $x,v\in \sX$ and $w\in\partial g_i(v,s)$, 
we show that $\ps{p_i(u,s) - w, \sz(u) - v} \geq 0$ for 
$du\otimes \mu_i$-almost all $(u,s)$. Since $v\in \sX$ and 
$w\in\partial g_i(v,s)$ are arbitrary, we get that 
$p_i(u,s) \in \partial g_i(\sz(u), s)$ by a well known property of the 
subdifferentials of $\Gamma_0(\sX)$ functions. 

\section{Conclusion} 
\label{sec-conc} 

A fast regularized optimization algorithm over large unstructured graphs was
introduced in this paper. This algorithm is a variant of the proximal gradient
algorithm that operates on randomly chosen simple paths. It belongs to the
family of stochastic approximation algorithms with a decreasing step size. One
future research direction consists in a fine convergence analysis of this
algorithm, hopefully leading to a provably optimal choice of the total walk
length $L$. Another research direction concerns the constant step analogue of
the described algorithm, whose transient behavior could be interesting in
many applicative contexts in the fields of statistics and learning.

\appendices 

\section{Proofs for Sec.~\ref{sec:prf-main}}
\label{anx-prf-main} 

\subsection{Proof of Lem.~\ref{bnd-nf-ng}} 
\label{prf-nf-ng} 
We start by writing 
$\| \nabla f_i(\bar x^{i-1},s^i) \| \leq \| \nabla f_i(\bar x^{i-2},s^i) \| 
+ K_i(s^i) \| \bar x^{i-1} - \bar x^{i-2}\|$, where $K_i(s^i)$ is provided by
Ass.~\ref{f-Lip}. Using the identity 
$\bar x^{i-1} = \sT_{\gamma,i-1}(\bar x^{i-2})$, where $\sT_{\gamma,i}$ is 
given by~\eqref{Tyosida}, and recalling that $\nabla g^\gamma_i(\cdot,s^i)$ is 
$\gamma^{-1}$-Lipschitz, we get 
\begin{multline*} 
\| \nabla f_i(\bar x^{i-1},s^i) \| \leq 
  \| \nabla f_i(\bar x^{i-2},s^i) \| \\ 
  + \gamma K_i(s^i) (2 \| \nabla f_{i-1}(\bar x^{i-2},s^{i-1}) \|  
+ \| \nabla g_{i-1}^\gamma(\bar x^{i-2},s^{i-1}) \| ) . 
\end{multline*} 
Similarly, 
\begin{align*} 
& \| \nabla g_i^\gamma(\bar x^{i-1} 
    - \gamma \nabla f_i(\bar x^{i-1},s^i), s^i) \| \\
&\leq 
\| \nabla f_i(\bar x^{i-1},s^i ) \| + 
2 \| \nabla f_{i-1}(\bar x^{i-2},s^{i-1} ) \| \\
&\phantom{=} + 
\| \nabla g_i^\gamma(\bar x^{i-2},s^i ) \| + 
\| \nabla g_{i-1}^\gamma(\bar x^{i-2},s^{i-1} ) \| . 
\end{align*}
Iterating down to $\bar x^0 = x$, we get the result since for every $i$, $K_i(\xi^i)$ admit all their moments.

\subsection{Proof of Prop.~\ref{fejer}} 
\label{anx-prf-fejer}

Let $x_\star$ be an arbitrary element of $\mZ$. Let $i \in \{1,\ldots,L\}$. We start by writing 
\begin{align*} 
\|\bar x_{n+1}^{i} - x_\star \|^2 &= \| \bar x_{n+1}^{i} - \bar x_{n+1}^{i-1} \|^2 
 + \| \bar x_{n+1}^{i-1} - x_\star \|^2 \\
&\phantom{=} + 2 \ps{ \bar x_{n+1}^{i} - \bar x_{n+1}^{i-1}, 
\bar x_{n+1}^{i-1} - x_\star} \\
&= \| \bar x_{n+1}^{i-1} - x_\star \|^2 + 
  \gamma^2 \| \nabla f_i + \nabla g^\gamma_i \|^2 \\ 
&\phantom{=}
-2\gamma\ps{\nabla f_i - \nabla f_i^\star, \bar x_{n+1}^{i-1} - x_\star} \\
&\phantom{=}
-2\gamma\ps{\nabla g_i^\gamma - \varphi_i, \bar x_{n+1}^{i-1} - x_\star} \\
&\phantom{=}
-2\gamma\ps{\nabla f_i^\star + \varphi_i, \bar x_{n+1}^{i-1} - x_\star} \\
&= \| \bar x_{n+1}^{i-1} - x_\star \|^2 + A_1 + A_2 + A_3 + A_4 . 
\end{align*} 
Most of the proof consists in bounding the $A_i$'s.  For an arbitrary
$\beta > 0$, it holds that  
$|\ps{a,b}|\leq (\beta/2) \|a\|^2 + \|b\|^2 / (2\beta)$.  We rewrite this
inequality as $|\ps{a,b}|\leq \eta \|a\|^2 + C \|b\|^2$, where $\eta > 0$ is a
constant chosen as small as desired, and $C > 0$ is fixed accordingly.
Similarly, $\| a + b \|^2 \leq (1+\eta) \| a \|^2 + C \|b \|^2$ where 
$\eta > 0$ is as small as desired.  We shall repeatedly use these inequalities
without mention. Similarly to $C$, the value of $\eta$ can change from a line 
of calculation to another. 

Starting with $A_1$, we have 
\[
A_1 \leq 
  \gamma^2(1+\eta) \| \nabla g^\gamma_i \|^2 
  + C \gamma^2 \| \nabla f_i \|^2 . 
\] 
We have $A_2\leq 0$ by the convexity of $f_L$. We can write 
\begin{align*}
A_3 &= 
-2\gamma\ps{\nabla g_i^\gamma - \varphi_i, \prox_{\gamma g_i} - x_\star} \\
&\phantom{=} -2\gamma\ps{\nabla g_i^\gamma - \varphi_i, \bar x_{n+1}^{i-1} - 
  \gamma \nabla f_i - \prox_{\gamma g_i} } \\
&\phantom{=} -2\gamma\ps{\nabla g_i^\gamma - \varphi_i, \gamma \nabla f_i } 
\end{align*} 
By \eqref{yos-mon}, the first term at the right hand side is $\leq 0$. 
By \eqref{yosida}, 
$\bar x_{n+1}^{i-1} - \gamma \nabla f_i - \prox_{\gamma g_i} = 
\gamma \nabla g_i^\gamma$. Thus, 
\begin{align*}
A_3 &\leq 
-2 \gamma^2 \| \nabla g_i^\gamma \|^2 
+ 2\gamma^2 \ps{\varphi_i, \nabla g_i^\gamma + \nabla f_i } 
- 2\gamma^2 \ps{\nabla g_i^\gamma, \nabla f_i} \\
& \leq -(2-\eta) \gamma^2 \| \nabla g_i^\gamma \|^2 
+ C\gamma^2 \| \nabla f_i \|^2 + C\gamma^2 \|\varphi_i \|^2  
\end{align*} 
As regards $A_4$, we have 
\begin{align*}
A_4 &= -2\gamma\ps{\nabla f_i^\star + \varphi_i, x_{n} - x_\star} \\
& \phantom{=} - 2\gamma \ps{\nabla f_i^\star + \varphi_i, \bar x_{n+1}^{i-1} - x_{n}}.
\end{align*}
Gathering these inequalities, we get 
\begin{align*} 
\|\bar x_{n+1}^{i} - x_\star \|^2 & \leq \| \bar x_{n+1}^{i-1} - x_\star \|^2 
- (1-\eta) \gamma^2 \| \nabla g^\gamma_i \|^2\\ 
 & \phantom{=} + C \gamma^2 \| \nabla f_i \|^2 + C\gamma^2 \|\varphi_i \|^2 \\
&\phantom{=} -2\gamma\ps{\nabla f_i^\star + \varphi_i, x_n - x_\star}\\
&\phantom{=} - 2\gamma \ps{\nabla f_i^\star + \varphi_i, \bar x_{n+1}^{i-1} - x_{n}}
\end{align*} 
Iterating over $i$, we get 
\begin{align*} 
\|\bar x_{n+1}^{i} - x_\star \|^2 & \leq \| x_n - x_\star \|^2 
- (1-\eta) \gamma^2 \sum_{k=1}^{i} \| \nabla g^\gamma_k \|^2\\ 
 & \phantom{=} + C \gamma^2 \sum_{k=1}^{i} \| \nabla f_k \|^2 + C\gamma^2 \sum_{k=1}^{i} \|\varphi_k \|^2 \\
&\phantom{=} -2\gamma \sum_{k=1}^{i} \ps{\nabla f_k^\star + \varphi_k, x_n - x_\star}\\
&\phantom{=} -2\gamma \sum_{k=1}^{i} \ps{\nabla f_k^\star + \varphi_k, \bar x_{n+1}^{k-1} - x_n} . 
\end{align*} 
The summand in the last term can be written as 
\begin{align*}
& -2\gamma\ps{\nabla f_k^\star + \varphi_k, \bar x_{n+1}^{k-1} - x_n}\\
=& -2\gamma \sum_{\ell=1}^{k-1} \ps{\nabla f_k^\star + \varphi_k, \bar x_{n+1}^{\ell} - \bar x_{n+1}^{\ell-1}}\\
=& -2\gamma^2 \sum_{\ell=1}^{k-1} \ps{\nabla f_k^\star + \varphi_k, \nabla f_\ell + \nabla g^\gamma_\ell}\\
\phantom{=} & \leq \gamma^2 C \|\nabla f_k^\star\|^2 + \gamma^2 C\|\varphi_k\|^2 \\
&\phantom{=} + \gamma^2 C \sum_{\ell=1}^{k-1} \|\nabla f_\ell\|^2 + \gamma^2 \eta \sum_{\ell=1}^{k-1} \|\nabla g^\gamma_\ell\|^2. 
\end{align*}
where we used $|\ps{a,b}| \leq \eta\|a\|^2 + C\|b\|^2$ as above.
Therefore, for all $i = 1,\ldots,L,$
\begin{align} 
\notag
\|\bar x_{n+1}^{i} - x_\star \|^2 & \leq \| x_n - x_\star \|^2 
 - (1-\eta) \gamma^2 \sum_{k=1}^{i} \| \nabla g^\gamma_k \|^2\\ \notag
 & \phantom{=} + C \gamma^2 \sum_{k=1}^{i} \| \nabla f_k^\star \|^2 + C\gamma^2 \sum_{k=1}^{i} \|\varphi_k \|^2 \\ \notag
&\phantom{=} + C \gamma^2 \sum_{k=1}^{i} \| \nabla f_k \|^2 \\ 
&\phantom{=} -2\gamma \ps{\sum_{k=1}^{i} \nabla f_k^\star + \varphi_k, x_n - x_\star}.
\label{eq:i}
\end{align}
We consider the case $i = L$. Using Ass.~\ref{selec}, 
\begin{align*} 
\bE_n \left[ \|\bar x_{n+1}^{L} - x_\star \|^2 \right] & \leq \| x_n - x_\star \|^2 \\
& \phantom{=} - (1-\eta) \gamma^2 \bE_n \left[ \sum_{k=1}^{L} \| \nabla g^\gamma_k \|^2 \right]
 \\
&\phantom{=} + C \gamma^2 + C \gamma^2 \sum_{k=1}^{L} \bE_n [\| \nabla f_k \|^2 ]\\
&\phantom{=} -2\gamma \bE_n \left[ \ps{\sum_{k=1}^{L} \nabla f_k^\star + \varphi_k, x_n - x_\star} \right].
\end{align*}
The last term at the right hand side is zero since 
\begin{align*}
&\bE_n \left[ \ps{\sum_{k=1}^{L} \nabla f_k^\star + \varphi_k, x_n - x_\star} \right] \\
=& \ps{\bE \left[ \sum_{k=1}^{L} \nabla f_k^\star + \varphi_k \right] , x_n - x_\star} = 0
\end{align*}
by definition of $\nabla f_k^\star$ and $\varphi_k$. 
Besides, using Ass.~\ref{f-Lip}, for all $k$ we have $$\bE_{n}[\| \nabla f_k \|^2] \leq C \bE_{n}[\| \nabla f_k^\star \|^2] + C \bE_n [K^2_k(\xi_{n+1}^{k})\|\bar x_{n+1}^{k-1} - x_\star\|^2]. $$ 
Then,
\begin{align} 
\notag
\bE_n [\| x_{n+1} - x_\star \|^2] & \leq \| x_n - x_\star \|^2 + C \gamma^2 \\ \notag
&\phantom{=} - (1-\eta) \gamma^2 \bE_n \left[\sum_{k=1}^{L} \| \nabla g^\gamma_k \|^2 \right]\\ 
&\phantom{=} + C \gamma^2 \sum_{k=1}^{L} \bE_n \left[ K^2_k(\xi_{n+1}^{k})\|\bar x_{n+1}^{k-1} - x_\star\|^2 \right]
\label{eq:L}
\end{align}
We shall prove by induction that for all r.v $P_{k}$ which is a monomial expression of the r.v $K^2_k(\xi_{n+1}^{k}),\ldots,K^2_L(\xi_{n+1}^{L})$, there exists $C > 0$ such that 
\begin{equation}
\bE_n \left[ P_{k} \|\bar x_{n+1}^{k-1} - x_\star\|^2 \right] \leq C(1 + \|x_n - x_\star\|^2),
\label{eq:induction}
\end{equation}
for all $k = 1,\ldots,L.$ Note that such a r.v $P_k$ is independent of $\mcF_n$, non-negative and for all $\alpha > 0$, $\bE[P_k^\alpha]<\infty$ by Ass.~\ref{f-Lip}. Using Ass.~\ref{f-Lip}, the induction hypothesis~\ref{eq:induction} is satisfied if $k = 1$. Assume that it holds true until the step $k-1$ for some $k \leq L$. Using~\ref{eq:i} and Ass.~\ref{f-Lip},   
\begin{align}
\notag
 \bE_n \left[ P_{k}\|\bar x_{n+1}^{k-1} - x_\star \|^2 \right]
& \leq C\| x_n - x_\star \|^2 \\ \notag
& \phantom{=} + C\gamma^2 \bE_n \left[ P_{k} \sum_{\ell=1}^{k-1} \| \nabla f_\ell \|^2 \right]\\ \notag
 & \phantom{=} + C\gamma^2 \bE_n \left[ P_{k} \sum_{\ell=1}^{k-1} \|\varphi_\ell \|^2 + \| \nabla f_\ell^\star \|^2 \right] \\ 
&\phantom{=} -2 \gamma \bE_n P_k  \ps{\sum_{\ell=1}^{k-1} \nabla f_\ell^\star + \varphi_\ell, x_n - x_\star}. 
\label{eq:heredite}
\end{align}
The last term at the right hand side can be bounded as
\begin{align}
\notag
&-2\gamma \bE_n P_k \ps{\sum_{\ell=1}^{k-1} \nabla f_\ell^\star + \varphi_\ell, x_n - x_\star} \\ \notag
\leq& C \|x_n - x_\star\|^2 + C \bE_n \left[ P_k \sum_{\ell=1}^{k-1} \|\nabla f_\ell^\star\|^2 + \|\varphi_\ell \|^2 \right] \\ 
\leq& C \|x_n - x_\star\|^2 + C
\label{eq:heredite1}
\end{align}
using Hölder inequality and Ass.~\ref{selec}. 
For all $\ell = 1,\ldots,k-1$,
\begin{align}
\notag
\bE_{n}[P_k \| \nabla f_\ell \|^2] &\leq C \bE_{n}[P_k \| \nabla f_\ell^\star \|^2] \\ \notag
&\phantom{=} + C \bE_n \left[ P_k K^2_\ell(\xi_{n+1}^{\ell})\|\bar x_{n+1}^{\ell-1} - x_\star\|^2 \right] \\
&\leq C(1 + \|x_n - x_\star\|^2)
\label{eq:heredite2}
\end{align}
where we used Hölder inequality and Ass.~\ref{selec} for the first term at the right hand side and the induction hypothesis~\eqref{eq:induction} at the step $\ell$ with the r.v $P_\ell := P_k K^2_\ell(\xi_{n+1}^{\ell})$ for the second term.

Plugging~\eqref{eq:heredite1} and~\eqref{eq:heredite2} into~\eqref{eq:heredite} and using again Hölder inequality and Ass.~\ref{selec} we find that~\eqref{eq:induction} holds true at the step $k$. Hence~\eqref{eq:induction} holds true for all $k = 1,\ldots,L$. 
Finally, plugging~\eqref{eq:induction} into~\eqref{eq:L} with $P_k = K^2_k(\xi_{n+1}^{k})$ for all $k = 1,\ldots,L$ we get
\begin{align*} 
\bE_n [\| x_{n+1} - x_\star \|^2] & \leq  (1 + C\gamma^2) \| x_n - x_\star \|^2 + C \gamma^2  \\  
 &\phantom{=} 
  - (1 - \eta) \gamma^2 \bE_n \left[\sum_{k=1}^{L} \| \nabla g^\gamma_k \|^2 \right]. 
\end{align*} 
By the Robbins-Siegmund lemma \cite{robbins1971convergence}, used along with 
$(\gamma_n) \in \ell^2$, we get that $(\| x_n - x_\star \|)$ converges almost
surely, showing the first point. 

By taking the expectations at both sides of this inequality, 
we also obtain that $(\bE \|x_n - x_\star \|^2)$ converges, 
$\sup_n \bE \|x_n - x_\star \|^2 < \infty$, 
and $\bE \sum_n \gamma_{n+1}^2 
  \sum_{i=1}^L \| \nabla g_{i}^\gamma \|^2 < \infty$.  
As $\sup_n \bE \|x_n - x_\star \|^2 < \infty$, we have by 
Ass.~\ref{f-Lip} that $\sup_n \bE \| \nabla f_1 \|^2 < \infty$. 
Using Lem.~\ref{bnd-nf-ng} and iterating, we easily get that 
$\bE \sum_n \gamma_{n+1}^2 
  \sum_{i=1}^L \| \nabla f_{i} \|^2 < \infty$ for all $i$.  

Since
$\| \bar x_{n+1}^1 - x_n \| \leq \gamma \| \nabla f_1 \| + 
\gamma \| \nabla g_1^\gamma  \|$, we get that 
$\sum_n \bE \| \bar x_{n+1}^1 - x_n \|^2 < \infty$. By 
Borel-Cantelli's lemma, we get that $\bar x_{n+1}^1 - x_n \to 0$ almost surely. 
The almost sure convergence of $\bar x_{n+1}^i - x_n$ to zero is 
shown similarly, and the proof of Prop.~\ref{fejer} is concluded.

\bibliographystyle{IEEEtran}
\bibliography{math}

\def\cprime{$'$} \def\cdprime{$''$} \def\cprime{$'$}
\begin{thebibliography}{10}
\providecommand{\url}[1]{#1}
\csname url@samestyle\endcsname
\providecommand{\newblock}{\relax}
\providecommand{\bibinfo}[2]{#2}
\providecommand{\BIBentrySTDinterwordspacing}{\spaceskip=0pt\relax}
\providecommand{\BIBentryALTinterwordstretchfactor}{4}
\providecommand{\BIBentryALTinterwordspacing}{\spaceskip=\fontdimen2\font plus
\BIBentryALTinterwordstretchfactor\fontdimen3\font minus
  \fontdimen4\font\relax}
\providecommand{\BIBforeignlanguage}[2]{{%
\expandafter\ifx\csname l@#1\endcsname\relax
\typeout{** WARNING: IEEEtran.bst: No hyphenation pattern has been}%
\typeout{** loaded for the language `#1'. Using the pattern for}%
\typeout{** the default language instead.}%
\else
\language=\csname l@#1\endcsname
\fi
#2}}
\providecommand{\BIBdecl}{\relax}
\BIBdecl

\bibitem{el2016asymptotic}
A.~El~Alaoui, X.~Cheng, A.~Ramdas, M.~J. Wainwright, and M.~I. Jordan,
  ``Asymptotic behavior of $ \ell_p $-based {L}aplacian regularization in
  semi-supervised learning,'' in \emph{COLT}, 2016, pp. 879--906.

\bibitem{zhu2003semi}
X.~Zhu, Z.~Ghahramani, and J.~Lafferty, ``Semi-supervised learning using
  gaussian fields and harmonic functions,'' in \emph{ICML}, 2003, pp. 912--919.

\bibitem{hallac2015network}
D.~Hallac, J.~Leskovec, and S.~Boyd, ``Network lasso: Clustering and
  optimization in large graphs,'' in \emph{SIGKDD}, 2015, pp. 387--396.

\bibitem{chambolle2010introduction}
A.~Chambolle, V.~Caselles, D.~Cremers, M.~Novaga, and T.~Pock, ``An
  introduction to total variation for image analysis,'' \emph{Theoretical
  foundations and numerical methods for sparse recovery}, vol.~9, pp. 263--340,
  2010.

\bibitem{hinterberger2003tube}
W.~Hinterberger, M.~Hinterm{\"u}ller, K.~Kunisch, M.~Von~Oehsen, and
  O.~Scherzer, ``Tube methods for bv regularization,'' \emph{Journal of
  Mathematical Imaging and Vision}, vol.~19, no.~3, pp. 219--235, 2003.

\bibitem{harchaoui2012multiple}
Z.~Harchaoui and C.~L{\'e}vy-Leduc, ``Multiple change-point estimation with a
  total variation penalty,'' \emph{Journal of the American Statistical
  Association}, 2012.

\bibitem{tibshirani2014adaptive}
R.~J. Tibshirani, ``Adaptive piecewise polynomial estimation via trend
  filtering,'' \emph{The Annals of Statistics}, vol.~42, no.~1, pp. 285--323,
  2014.

\bibitem{wang2014trend}
Y.-X. Wang, J.~Sharpnack, A.~Smola, and R.~J. Tibshirani, ``Trend filtering on
  graphs,'' \emph{Journal of Machine Learning Research}, vol.~17, no. 105, pp.
  1--41, 2016.

\bibitem{padilla2016dfs}
O.~H.~M. Padilla, J.~G. Scott, J.~Sharpnack, and R.~J. Tibshirani, ``The dfs
  fused lasso: nearly optimal linear-time denoising over graphs and trees,''
  \emph{arXiv preprint arXiv:1608.03384}, 2016.

\bibitem{hutter2016optimal}
J.-C. H{\"u}tter and P.~Rigollet, ``Optimal rates for total variation
  denoising,'' \emph{arXiv preprint arXiv:1603.09388}, 2016.

\bibitem{landrieu2016cut}
L.~Landrieu and G.~Obozinski, ``Cut pursuit: Fast algorithms to learn piecewise
  constant functions,'' in \emph{AISTATS}, 2016, pp. 1384--1393.

\bibitem{tansey2015fast}
W.~Tansey and J.~G. Scott, ``A fast and flexible algorithm for the graph-fused
  lasso,'' \emph{arXiv preprint arXiv:1505.06475}, 2015.

\bibitem{barberoTV14}
A.~Barbero and S.~Sra, ``Modular proximal optimization for multidimensional
  total-variation regularization,'' \emph{arXiv preprint arXiv:1411.0589},
  2014.

\bibitem{ben2015robust}
W.~Ben-Ameur, P.~Bianchi, and J.~Jakubowicz, ``Robust distributed consensus
  using total variation,'' \emph{IEEE Transactions on Automatic Control},
  vol.~61, no.~6, pp. 1550--1564, 2016.

\bibitem{chen2014signal}
S.~Chen, A.~Sandryhaila, G.~Lederman, Z.~Wang, J.~M. Moura, P.~Rizzo,
  J.~Bielak, J.~H. Garrett, and J.~Kova{\v{c}}evi{\'c}, ``Signal inpainting on
  graphs via total variation minimization,'' in \emph{ICASSP}, 2014, pp.
  8267--8271.

\bibitem{condat2013direct}
L.~Condat, ``A direct algorithm for 1d total variation denoising,'' \emph{IEEE
  SPL}, vol.~20, no.~11, pp. 1054--1057, 2013.

\bibitem{graham1997spectral}
F.~R. Chung, \emph{Spectral graph theory}.\hskip 1em plus 0.5em minus
  0.4em\relax American Mathematical Soc., 1997, vol.~92.

\bibitem{spielman2010algorithms}
D.~A. Spielman, ``Algorithms, graph theory, and linear equations in laplacian
  matrices,'' in \emph{Proceedings of the ICM}, vol.~4, 2010, pp. 2698--2722.

\bibitem{bottou2010large}
L.~Bottou, ``Large-scale machine learning with stochastic gradient descent,''
  in \emph{COMPSTAT'2010}, 2010, pp. 177--186.

\bibitem{bottou2016optimization}
L.~Bottou, F.~E. Curtis, and J.~Nocedal, ``Optimization methods for large-scale
  machine learning,'' \emph{arXiv preprint arXiv:1606.04838}, 2016.

\bibitem{bia-hac-16}
P.~Bianchi and W.~Hachem, ``Dynamical behavior of a stochastic
  forward--backward algorithm using random monotone operators,'' \emph{Journal
  of Optimization Theory and Applications}, vol. 171, no.~1, pp. 90--120, 2016.

\bibitem{johnson2013dynamic}
N.~A. Johnson, ``A dynamic programming algorithm for the fused lasso and l
  0-segmentation,'' \emph{Journal of Computational and Graphical Statistics},
  vol.~22, no.~2, pp. 246--260, 2013.

\bibitem{mammen1997locally}
E.~Mammen and S.~van~de Geer, ``Locally adaptive regression splines,''
  \emph{The Annals of Statistics}, vol.~25, no.~1, pp. 387--413, 1997.

\bibitem{davies2001local}
P.~L. Davies and A.~Kovac, ``Local extremes, runs, strings and
  multiresolution,'' \emph{The Annals of Statistics}, pp. 1--48, 2001.

\bibitem{combettes2009iterative}
P.~L. Combettes, ``Iterative construction of the resolvent of a sum of maximal
  monotone operators,'' \emph{Journal of Convex Analysis}, vol.~16, no.~4, pp.
  727--748, 2009.

\bibitem{jegelka2013reflection}
S.~Jegelka, F.~Bach, and S.~Sra, ``Reflection methods for user-friendly
  submodular optimization,'' in \emph{Advances in NIPS}, 2013, pp. 1313--1321.

\bibitem{spielman2014nearly}
D.~A. Spielman and S.-H. Teng, ``Nearly linear time algorithms for
  preconditioning and solving symmetric, diagonally dominant linear systems,''
  \emph{SIAM Journal on Matrix Analysis and Applications}, vol.~35, no.~3, pp.
  835--885, 2014.

\bibitem{roc-69(mes)}
R.~T. Rockafellar, ``Measurable dependence of convex sets and functions on
  parameters,'' \emph{Journal of Mathematical Analysis and Applications},
  vol.~28, no.~1, pp. 4--25, 1969.

\bibitem{pas-79}
G.~B. Passty, ``Ergodic convergence to a zero of the sum of monotone operators
  in hilbert space,'' \emph{Journal of Mathematical Analysis and Applications},
  vol.~72, no.~2, pp. 383--390, 1979.

\bibitem{bia-16}
P.~Bianchi, ``Ergodic convergence of a stochastic proximal point algorithm,''
  \emph{SIAM Journal on Optimization}, vol.~26, no.~4, pp. 2235--2260, 2016.

\bibitem{wan-ber-15}
M.~Wang and D.~P. Bertsekas, ``Incremental constraint projection methods for
  variational inequalities,'' \emph{Mathematical Programming}, vol. 150, no.~2,
  pp. 321--363, 2015.

\bibitem{roc-wet-82}
R.~T. Rockafellar and R.~J. Wets, ``On the interchange of subdifferentiation
  and conditional expectation for convex functionals,'' \emph{Stochastics: An
  International Journal of Probability and Stochastic Processes}, vol.~7,
  no.~3, pp. 173--182, 1982.

\bibitem{snapnets}
J.~Leskovec and A.~Krevl, ``{SNAP Datasets}: {Stanford} large network dataset
  collection,'' \url{http://snap.stanford.edu/data}, Jun. 2014.

\bibitem{holland1983stochastic}
P.~W. Holland, K.~B. Laskey, and S.~Leinhardt, ``Stochastic blockmodels: First
  steps,'' \emph{Social networks}, vol.~5, no.~2, pp. 109--137, 1983.

\bibitem{l-bfgs-b(95)}
R.~H. Byrd, P.~Lu, J.~Nocedal, and C.~Y. Zhu, ``A limited memory algorithm for
  bound constrained optimization,'' \emph{SIAM Journal on Scientific
  Computing}, vol.~16, no.~5, pp. 1190--1208, 1995.

\bibitem{bau-com-livre11}
\BIBentryALTinterwordspacing
H.~H. Bauschke and P.~L. Combettes, \emph{Convex analysis and monotone operator
  theory in {H}ilbert spaces}, ser. CMS Books in Mathematics/Ouvrages de
  Math\'ematiques de la SMC.\hskip 1em plus 0.5em minus 0.4em\relax New York:
  Springer, 2011. [Online]. Available:
  \url{http://dx.doi.org/10.1007/978-1-4419-9467-7}
\BIBentrySTDinterwordspacing

\bibitem{moulines2011non}
E.~Moulines and F.~R. Bach, ``Non-asymptotic analysis of stochastic
  approximation algorithms for machine learning,'' in \emph{Advances in NIPS},
  2011, pp. 451--459.

\bibitem{vishnoi2012laplacian}
N.~K. Vishnoi, ``Laplacian solvers and their algorithmic applications,''
  \emph{Theoretical Computer Science}, vol.~8, no. 1-2, pp. 1--141, 2012.

\bibitem{bre-livre73}
\BIBentryALTinterwordspacing
H.~Br\'ezis, \emph{{Op\'erateurs maximaux monotones et semi-groupes de
  contractions dans les espaces de Hilbert}}, ser. North-Holland mathematics
  studies.\hskip 1em plus 0.5em minus 0.4em\relax Burlington, MA: Elsevier
  Science, 1973. [Online]. Available: \url{http://cds.cern.ch/record/1663074}
\BIBentrySTDinterwordspacing

\bibitem{aub-cel-(livre)84}
\BIBentryALTinterwordspacing
J.-P. Aubin and A.~Cellina, \emph{Differential inclusions}, ser. Grundlehren
  der Mathematischen Wissenschaften [Fundamental Principles of Mathematical
  Sciences].\hskip 1em plus 0.5em minus 0.4em\relax Springer-Verlag, Berlin,
  1984, vol. 264, set-valued maps and viability theory. [Online]. Available:
  \url{http://dx.doi.org/10.1007/978-3-642-69512-4}
\BIBentrySTDinterwordspacing

\bibitem{ben-hir-96}
M.~Bena{\"\i}m and M.~W. Hirsch, ``Asymptotic pseudotrajectories and chain
  recurrent flows, with applications,'' \emph{Journal of Dynamics and
  Differential Equations}, vol.~8, no.~1, pp. 141--176, 1996.

\bibitem{robbins1971convergence}
H.~Robbins and D.~Siegmund, ``A convergence theorem for non negative almost
  supermartingales and some applications,'' in \emph{Optimizing Methods in
  Statistics}.\hskip 1em plus 0.5em minus 0.4em\relax Academic Press, New York,
  1971, pp. 233--257.

\end{thebibliography}

\begin{IEEEbiography}[{\includegraphics[width=1in,height =1.25in,clip,keepaspectratio]{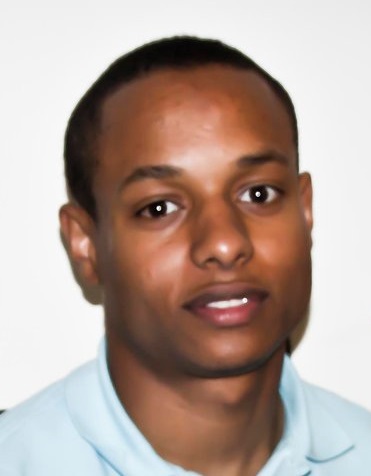}}]{Adil Salim} 
was born in 1991 in L'Hay-les-Roses, France. He received the M.Sc. degree of the University of Paris XI and the ENSAE ParisTech in 2015. Then he joined Telecom ParisTech – Université Paris-Saclay as a Ph.D. student in the Image, Data, Signal group. His research interests are focused
on stochastic approximation and monotone operator theory with special
emphasis on stochastic optimization algorithms and online learning algorithms.
\end{IEEEbiography} 

\begin{IEEEbiography}[{\includegraphics[width=1in,height =1.25in,clip,keepaspectratio]{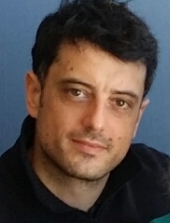}}]{Pascal Bianchi} 
was born in 1977 in Nancy, France. He received the M.Sc. degree of the University of Paris XI and Supélec in 2000 and the Ph.D. degree of the University of Marne-la-Vallée in 2003. From 2003 to 2009, he was with the Telecommunication Department of Centrale-Supélec. He is now working as a full Professor in the Statistics and Applications group at Telecom ParisTech – Université Paris-Saclay. His current research interests are in the area of numerical optimization, stochastic approximations, signal processing and distributed systems.
\end{IEEEbiography}

\begin{IEEEbiography}[{\includegraphics[width=1in,height =1.25in,clip,keepaspectratio]{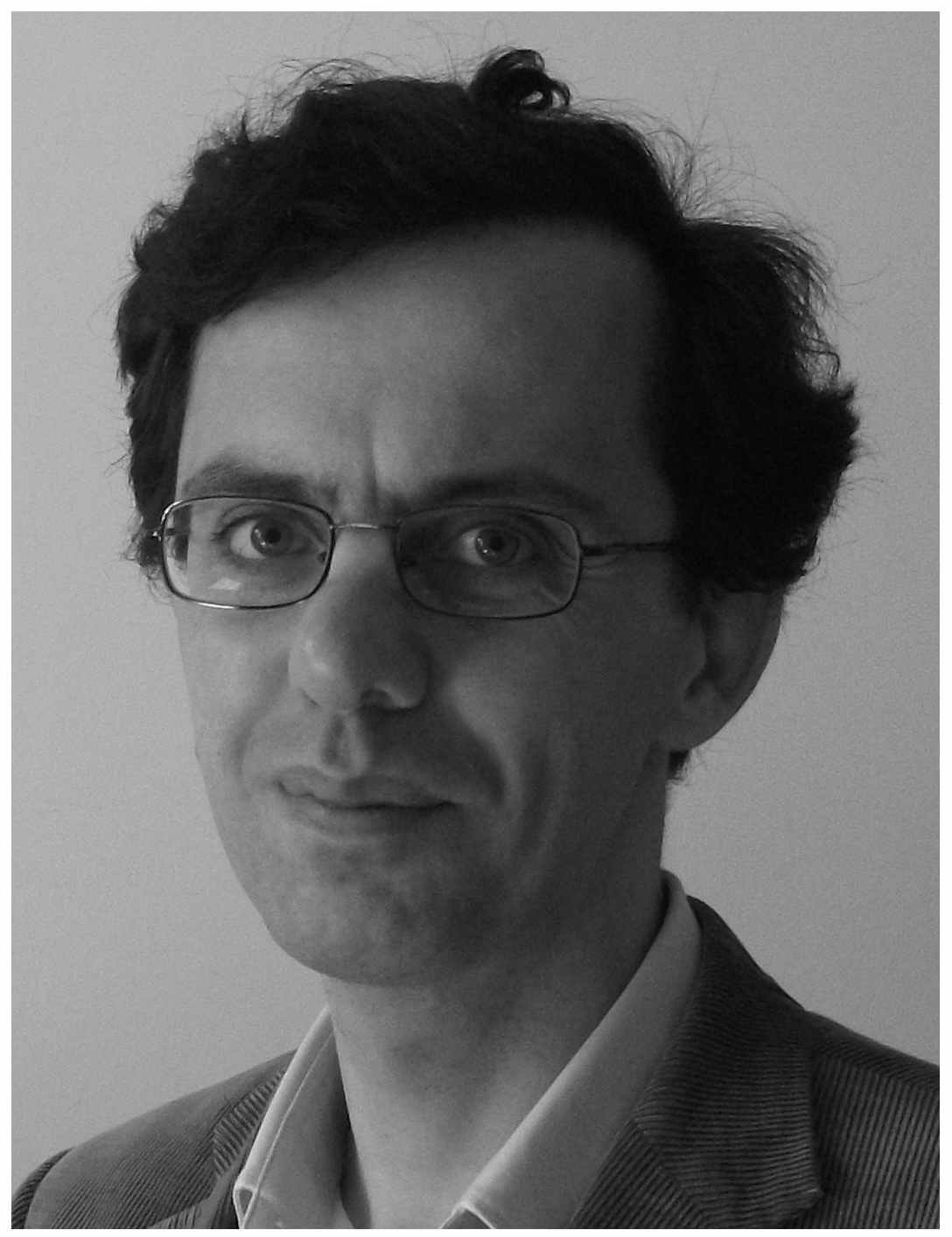}}]{Walid Hachem} 
was born in Bhamdoun, Lebanon, in 1967. He received the Engineering degree in
telecommunications from St Joseph University (ESIB), Beirut, Lebanon, in 1989,
the Masters degree from Telecom ParisTech, France, in 1990, the PhD degree in
signal processing from the Universit\'e Paris-Est Marne-la-Vall\'ee in 2000 
and the Habilitation \`a diriger des recherches from the Universit\'e Paris-Sud in 2006. \\
Between 1990 and 2000 he worked in the telecommunications industry as a signal
processing engineer. In 2001 he joined the academia as a faculty member at
Sup\'elec, France. In 2006, he joined the CNRS (Centre national de la Recherche
Scientifique), where he is now a research director based at the Universit\'e 
Paris-Est. 
His research themes consist mainly in the large random matrix theory and its
applications in statistical estimation and in communication theory, and in the
optimization algorithms in random environments.  \\ 
He served as an associate editor for the IEEE Transactions on Signal Processing
between 2007 and 2010. 
\end{IEEEbiography}

\end{document}